\documentclass[13pt]{amsart}
\usepackage{amsmath}
\usepackage{}
\usepackage{amsfonts}
\usepackage{amsfonts,latexsym,rawfonts,amsmath,amssymb,amsthm}
\usepackage[plainpages=false]{hyperref}

\usepackage{pdfsync}

\usepackage[a4paper]{geometry}
\usepackage{esint}
\usepackage{graphics,graphicx}
\usepackage{tikz}

\numberwithin{equation}{section}

\newcommand{\beq}{\begin{equation}}
\newcommand{\eeq}{\end{equation}}
\newcommand{\beqs}{\begin{eqnarray*}}
\newcommand{\eeqs}{\end{eqnarray*}}
\newcommand{\beqn}{\begin{eqnarray}}
\newcommand{\eeqn}{\end{eqnarray}}
\newcommand{\beqa}{\begin{array}}
\newcommand{\eeqa}{\end{array}}

\def\lra{\longrightarrow}

\def\bc{\begin{center}}
\def\ec{\end{center}}

\def\begeq{\begin{equation}}
\def\endeq{\end{equation}}
\def\and{\quad{\rm and}\quad}

\let\lra=\longrightarrow

\def\mapright\#1{\,\smash{\mathop{\lra}\limits^{\#1}}\,}

\newtheorem{prop}{Proposition}[section]
\newtheorem{theo}[prop]{Theorem}
\newtheorem{lem}[prop]{Lemma}

\newtheorem{cor}[prop]{Corollary}
\newtheorem{rem}[prop]{Remark}
\newtheorem{ex}[prop]{Example}

\begin{document}
\title{Horosymmetric  limits of K\"ahler-Ricci flow on Fano $\mathbf G$-manifolds }
\author{Gang $\text{Tian}^{\dag}$ and Xiaohua $\text{Zhu}^{\ddag}$}

\address{BICMR and SMS, Peking
University, Beijing 100871, China.}
\email{ tian@math.princeton.edu\\\ xhzhu@math.pku.edu.cn}

\thanks { $\dag$ $\ddag$ partially supported by National Key R\&D Program of China  2020YFA0712800.}


\begin{abstract} In this paper, we prove that on a Fano $\mathbf G$-manifold  $(M,J)$, the Gromov-Hausdorff limit of K\"ahler-Ricci flow
 with initial metric in $2\pi c_1(M)$ must be a $\mathbb Q$-Fano horosymmetric  variety $M_\infty$,  which   admits    a singular  K\"ahler-Ricci soliton.   Moreover,  $M_\infty$  is  a limit   of $\mathbb C^*$-degeneration of $M$  induced by an element in the Lie algebra of Cartan torus of $\mathbf G$.   A similar result can be also proved for K\"ahler-Ricci flows on any Fano horosymmetric manifolds. As an application, we generalize our previous result about the type II singularity of K\"ahler-Ricci flows  on  Fano  $\mathbf G$-manifolds to Fano horosymmetric manifolds.

\end{abstract}

\maketitle

\section{Introduction}

It has been known that the existence of K\"ahler-Einstein, abbreviated by KE, metrics on a Fano manifold $M$, that is a compact K\"ahler manifold with positive first Chern class, is equivalent to the K-stability (cf. \cite{Ti1, Ti2, CDS, BBJ, LTW, Li}, etc.). Since not every Fano manifold is K-stable, there are Fano manifolds which do not admit KE metrics and we are led to studying the problem on optimal deformations of such Fano manifolds which admit a canonical metric.
The K\"ahler-Ricci (KR) flow  provides an approach to solve this problem by geometric analytic method, more precisely, we have the following conjecture, referred as Hamilton-Tian (HT) conjecture (cf. \cite{Ti1, Pe}):

 {\it Any sequence of metrics $(M, \omega(t))$ in KR flow contains a subsequence converging to a length space $(M_\infty,\omega_\infty)$
 in the Gromov-Hausdorff  (GH) topology and $(M_\infty,\omega_\infty)$ is a smooth KR soliton outside a closed subset $S$, called the singular set, of codimension at least $4$.  Moreover, this subsequence of $(M, \omega(t))$ converges to $(M_\infty,\omega_\infty)$
 in the Cheeger-Gromov topology.}

 This conjecture has been solved (cf. \cite{Ti1, TZhang,  CW, Bam, WfZ1}). Actually, the existence of  GH limits follows from Perelman's noncollapsing result  \cite{Pe, ST} and Q. Zhang's non-expanding result \cite{Zhq}. Moreover, the uniqueness of $(M_\infty,\omega_\infty)$ is also true (see \cite{CSW, WfZ2, HL1}).

Recall that a KR soliton on a complex manifold $M$ is a pair $(X, \omega)$, where $X$ is a holomorphic vector field (HVF) on $M$ and $\omega$ is a K\"ahler metric on $M$, satisfying:
\begin{align}\label{kr-soliton}{\rm Ric}(\omega)\,-\,\omega\,=\,L_X(\omega),
\end{align}
where $L_X$ denotes the Lie derivative along $X$. If $X=0$, the KR soliton becomes a
KE metric. The uniqueness theorem in \cite{TZ1, TZ2} states that a KR
soliton on a compact complex manifold, if it exists, must be unique modulo ${\rm Aut}(M)$.\footnote {In the case of KE metrics, this uniqueness theorem is due to Bando-Mabuchi \cite{BM}.} Furthermore, $X$ lies in the center of Lie algebra of the reductive part of ${\rm Aut}(M)$.

On a Fano manifold $M$, we usually consider the following normalized KR flow,
\begin{align}\label{KRF}
\frac{\partial \omega(t)}{\partial t}\, =\, -{\rm Ric}(\omega( t))\, +\,\omega(t), ~\omega( 0)=\omega_0,
\end{align}
where $\omega_0$ and $\omega (t)$ denote the K\"ahler forms of a given K\"ahler metric $g_0$ and the solutions of Ricci flow, respectively.
It is proved in \cite{Cao} that (\ref{KRF}) has a global solution $\omega(t)$ for all $t\ge 0$ whenever $\omega_0$ represents $2\pi c_1(M)$.
It is a natural problem to study the limiting behavior of $\omega(t)$ as $t\to \infty$ as well as its limiting structure.

The purpose of this paper is to solve the above problem in case of Fano $\mathbf G$-manifolds, where $\mathbf G$ is a complex reductive Lie group (i.e., a complexification of compact group $\mathbf K$). By a $\mathbf G$-manifold, we mean a {\it (bi-equivariant) compactification of $\mathbf G$} which admits a holomorphic $\mathbf G\times \mathbf G$-action and has an open and dense orbit isomorphic to $\mathbf G$ as a $\mathbf G\times \mathbf G$-homogeneous space. A special case of $\mathbf G$-manifolds is the case of toric manifolds when $\mathbf G$ is a torus.

The existence problem of KE metrics and KR solitons on $\mathbf G$-manifolds has been extensively studied (cf. \cite{Del1, Del2, LZZ, LTZ2, DH}, etc.).
A criterion has been found for the existence in terms of the barycenter of moment polytope associated to the Cartan torus subgroup of $\mathbf G$. By using this criterion, one can construct many examples of $\mathbf G$-manifolds which admit KE metrics or KR solitons, as well as examples of $\mathbf G$-manifolds which do not admit neither  KE metrics nor KR solitons (cf. \cite{Del1, Zhu2}). Furthermore, we have recently proved the following result for the KR flow on $\mathbf G$-manifolds \cite{LTZ1, Zhu3}.

\begin{theo}\label{g-flow} Let $(M,J)$ be a Fano $\mathbf G$-manifold which admits no KR-soliton. Then any solution $\omega(t)$ of
(\ref{KRF}) with initial metric $\omega_0\in 2\pi c_1(M,J)$ will develop singularity of type II, that is, curvature of $\omega(t)$ must blow up as $t\to \infty$.

\end{theo}

By   Theorem  \ref{g-flow},  we found   that there are two $\mathrm{SO}_4(\mathbb{C})$-manifolds and one $\mathrm{Sp}_4(\mathbb{C})$-manifold on which the KR flow  develops singularities of type II \cite{LTZ1}.
The result  provides  the first example of Fano manifolds on which the KR flow develops singularity of type II.

 Theorem  \ref{g-flow}  implies that the GH limit $(M_\infty,\omega_\infty)$ from the HT conjecture must be singular \cite{Bam, WfZ2}.  Also we have shown that the limit may not be a Fano $\mathbf G$-variety  \cite[Section 7]{LTZ2}.  Nevertheless, we hope that $(M_\infty,\omega_\infty)$ will still keep some symmetries and can be classified. Our first main theorem of this paper is the following:

\begin{theo}\label{main-theorem}
 The GH  limit of KR flow    (\ref{KRF})  on  a Fano $\mathbf G$-manifold  $(M,J)$
 is a  $\mathbb Q$-Fano horosymmetric  variety,   which  admits a singular KR soliton.  More precisely, if an initial metric $\omega_0$ in (\ref{KRF}) is $\mathbf K\times \mathbf K$-invariant, the solution  $\omega(t)$ after  Cartan torus  transformations of  $\mathbf G$  converges  locally smoothly to a KR soliton  $\omega_{KS}$  on a horosymmetric space in Cheeger-Gromov topology,  whose completion is the  GH  limit of KR flow  (\ref{KRF}) with  a structure of  $\mathbb Q$-Fano horosymmetric variety  $M_\infty$ as a limit of  $\mathbb C^*$-degeneration of $(M,J)$  induced by an element in the Lie algebra of Cartan torus.  Moreover,
   $M_\infty$ is same with a   limit of $\mathbb C^*$-degeneration of $(M,J)$ induced by the soliton HVF of  $\omega_{KS}$,   if  $\omega_{KS}$ is not a KE metric.
\end{theo}

A homogeneous space $\mathbf G/\mathbf H$ is called horosymmetric if it is a fibration over a generalized flag  manifold whose fibers are symmetry spaces (cf. \cite {Del3}). A horosymmetric variety is simply a compactification  of a horosymmetric space $\mathbf G/\mathbf H$, if it is smooth, we call it a horosymmetric manifold.
Theorem \ref{main-theorem} means that any Fano $\mathbf G$-manifold has a degeneration to a $\mathbb Q$-Fano horosymmetric variety on which there is a singular KR soliton. Furthermore, this degeneration can be realized by a $\mathbb C^*$-degeneration  induced by an element in the Lie algebra of Cartan subgroup of $\mathbf G$
(cf. Theorem \ref{main-theorem-2}).   Namely, we have

\begin{cor}\label{optimal-degeneration} Any Fano $\mathbf G$-manifold $(M,J)$ admits a $\mathbf G\times \mathbf G$-equivariant $\mathbb C^*$-degeneration induced by an element in the Lie algebra of Cartan torus of    $\mathbf G$ such that
its central fiber in the $\mathbb C^*$-degeneration is modified $K$-stable relatively to the group $\mathbf G\times \mathbf G$.
\end{cor}

The modified K-stability in Corollary \ref{optimal-degeneration} follows from  the fact that  $(M_\infty, J_\infty)$  in Theorem \ref{main-theorem} admits  a singular KR soliton (cf. \cite{ BN, WZZ, HL2}).

\begin{rem} The $\mathbf G\times \mathbf G$ equivariant $\mathbb C^*$-degeneration in Corollary \ref{optimal-degeneration} is referred as an optimal degeneration for $(M,J)$.  Since the soliton HVF on $M_\infty$   in  Theorem  \ref{main-theorem}  induces a $\mathbf G\times \mathbf G$ equivariant $\mathbb C^*$-degeneration on $M$ which attains  the  minimum  of  $H$-invariant for special degenerations \cite{DS, WfZ2},  we can  find  a unique  degeneration  by minimizing  $H$-invariant for  the $\mathbf G\times \mathbf G$ equivariant  $\mathbb C^*$-degenerations   induced by the Lie algebra of Cartan torus  defined by Delcroix \cite{Del2}. Thus the algebraic  variety structure of   $M_\infty$    in Theorem \ref{main-theorem} can be classified by this way if $M$ is not K-semistable.   A detailed computation of $H$-invariant  on such degenerations has been recently given by Li-Li  \cite{LL}.

\end{rem}

Since the limit of  KR flow  (\ref{KRF}) is unique, we need to prove Theorem \ref{main-theorem} only for a $\mathbf K\times \mathbf K$-invariant initial metric. Then the flow can be reduced to a parabolic equation of  Monge-Amp\`ere (MA) type  for a class of convex functions as in the  case of toric manifolds \cite{Zhu1}.  The main difference here  is that  the convex functions   should be  invariant under the Weyl  subgroup of $\mathbf G$, but  the  induced  convex functions  by  Cartan torus  transformations will not preserve   Weyl-invariant in general.   In particular,  we can not get a global K\"ahler potentials  on a  $\mathbf G$-orbit by a convex function  after  torus  transformations to the $\mathbf K\times \mathbf K$-invariant  metric.   Our method is to do  local estimates  for  those induced  convex functions on an  Euclidean  cone $\mathfrak a_+$ and prove that  the limit convex function on a new   cone $\mathfrak a_+'$ will  be extended  to  define a KR soliton on a horosymmetric space (cf. Section  3-5).

We will divide our proof of Theorem \ref{main-theorem} into three cases according to  concentration  points  $x_t$ of convex functions $w_t$  on the cone $\mathfrak a_+$ (see Theorem \ref{main-theorem-2}).  Those points   induce a family of torus translations  which  will determine the soliton HVF on the limit space
$(M_\infty, J_\infty)$ (cf. Proposition \ref{limit-solution} and Lemma \ref{x-vector}).  We would like to mention that similar arguments have been done to get local estimates for K\"ahler potentials in KR flow in general by using the partial $C^0$-estimate (see  \cite{WfZ1}).

The horosymmetric space in Theorem \ref{main-theorem} can be explicitly constructed via the  $\mathbb C^*$-degeneration   induced by an element in the Lie algebra of Cartan subgroup of $\mathbf G$. (cf.  Example \ref{ex-degereration}).  For general construction of  $\mathbb C^*$-degenerations on horosymmetric varieties, we refer the reader to recent papers by Deltroix, Li-Li \cite{Del2, LL}.

Theorem \ref{main-theorem} can be generalized to any  horosymmetric manifolds.
Namely, we can also prove

 \begin{theo}\label{main-theorem-horosymmetric}
 The GH  limit of KR flow (\ref{KRF})  on  a Fano horosymmetric manifold $M=\overline {{\mathbf G}/{\mathbf H}}$
 is a $\mathbb Q$-Fano horosymmetric variety,  which admits a singular KR soliton.  More precisely, if an initial metric $\omega_0$ in (\ref{KRF}) is $\mathbf K$-invariant, the solution  $\omega(t)$ after Cartan torus transformations of $M$ converges to a KR soliton on a $\mathbb Q$-Fano horosymmetric variety in the Cheeger-Gromov topology,   whose completion is the  GH  limit of KR flow  (\ref{KRF}) with  a structure of  $\mathbb Q$-Fano horosymmetric variety  $M_\infty$ as a limit of  $\mathbb C^*$-degeneration of $(M,J)$  induced by an element in the Lie algebra of Cartan torus.  Moreover,
   $M_\infty$ is same with a   limit of $\mathbb C^*$-degeneration of $(M,J)$ induced by the soliton HVF of  $\omega_{KS}$,   if  $\omega_{KS}$ is not a KE metric.
\end{theo}

As an application,  we generalize   Theorem  \ref{g-flow} to the case of   Fano  horosymmetric manifolds as follows.

 \begin{theo}\label{g-flow-horo} Let $(M,J)$ be a Fano   horosymmetric  manifold which admits no KR-soliton. Then any solution $\omega(t)$ of
(\ref{KRF}) with initial metric $\omega_0\in 2\pi c_1(M,J)$ will develop singularity of type II.
\end{theo}

The proof of Theorem \ref{main-theorem-horosymmetric} is almost identical to that of  Theorem \ref{main-theorem}, so we will only sketch its proof in Section 6.

\section {$\mathbf K$-invariant  metrics on horosymmetry spaces}

In this paper,  we always assume that $\mathbf G$ is a reductive Lie group which is a complexification of compact Lie group $\mathbf K$. Let $\mathbf T^\mathbb C$ be an $r$-dimensional maximal complex torus of $\mathbf G$ with its Lie algebra $\mathfrak t^{\mathbb C}$ and $\mathfrak M$ the group of characters of $\mathfrak t^{\mathbb C}$. Denote the roots system of $(\mathbf G, \mathbf T^\mathbb C)$ in $\mathfrak M$ by $\Phi$ and choose a set of positive roots by $\Phi_+$.
Then each element in $\Phi$ can be regarded as the one of $\mathfrak a^*$, where $\mathfrak a^*$ is the dual of the non-compact part $\mathfrak a$ of $\mathfrak t^{\mathbb C}$.

\subsection{$\mathbf K\times \mathbf K$-invariant metrics on $\mathbf G$-manifolds }

 By the $\mathbf K\times \mathbf K$-invariance of K\"ahler metric  $\omega$ on $\mathbf G$,  the restriction of  $\omega$ on $\mathbf T^{\mathbb C}$ is an open toric K\"ahler metric.
Thus, it induces  a strictly convex  function $\psi$ on ${\mathfrak a}$  (also see Lemma \ref{Hessian} below)  such that
\begin{align}\label{convex-potential}\omega\,=\,\sqrt{-1}\partial\bar\partial \psi, ~{\rm on}~ \mathbf T^{\mathbb C}.
\end{align}

On the other hand, by the KAK-decomposition (\cite[Theorem 7.39]{Kn}), for any $g\in \mathbf G$,
there are $k_1,\,k_2\in \mathbf K$ and $x\in\mathfrak a$ such that $g=k_1\exp(x)k_2$. Here $x$ is uniquely determined up to a Weyl-action. This means that $x$ is unique in $\overline{\mathfrak a_+}$, where   $\overline{\mathfrak a_+}$ is  the close cone of
$\mathfrak a_+$ called  the Weyl chamber by
$$\mathfrak a_+=\{x\in \mathfrak a|~<\alpha, x>>0,~\forall ~\alpha\in\Phi_+\}.$$
Thus there is a bijection between $\mathbf K\times \mathbf K$-invariant functions $\phi$ on $\mathbf G$ and Weyl-invariant functions  $\psi$ on
$\mathfrak a$ which is given  by
$$\phi_\psi=\psi (\exp^{-1}(\cdot)):~{\mathbf T^{\mathbb C}}\to\mathbb R.$$
Hence,  (\ref{convex-potential}) can be extended on  $\mathbf G$ such  that
\begin{align}\label{convex-potential-global}\omega_\psi\,=\,\sqrt{-1}\partial\bar\partial \phi_\psi, ~{\rm on}~ \mathbf G.
\end{align}
 Without of  confusion,  we will not distinguish $\psi$ and $\phi_\psi$,  and  call $\phi_\psi$ (or $\psi$) convex on $G$ if  $\psi$  is   Weyl-invariant   convex on ${\mathfrak a}$.

The following KAK-integration formula can be found in \cite[Proposition 5.28]{Kn}.

\begin{prop}\label{KAK int}
Let $dV_{\mathbf G}$ be a Haar measure on $\mathbf G$ and $dx$ the Lebesgue measure
on $\mathfrak{a}$.
Then there exists a constant $C_H>0$ such that for any
$\mathbf K\times  \mathbf K$-invariant, $dV_{\mathbf G}$-integrable function $\psi$ on $\mathbf G$,
$$\int_{\mathbf G} \psi(g)\,dV_{\mathbf G}\,=\, C_H\,\int_{\mathfrak{a}_+}\psi(x){J}(x)\,dx,$$
where
\begin{align}\label{j-function} J(x)\,=\,\prod_{\alpha \in \Phi_+} \sinh^2\alpha(x).
\end{align}
\end{prop}

Without loss of generality, we may normalize $C_H=1$ for simplicity.

Next we recall a local holomorphic coordinate system on $\mathbf G$ used in \cite{Del1}. By the standard Cartan decomposition, we can decompose
$\mathfrak g$ as
$$\mathfrak g\,=\,\left(\mathfrak t\oplus\mathfrak a\right)\oplus\left(\oplus_{\alpha\in\Phi}V_{\alpha}\right),$$
where $\mathfrak t$ is the Lie algebra of  $T$ and
 $$V_{\alpha}\,=\,\{X\in\mathfrak g|~ad_H(X)\,=\,\alpha(H)X,~\forall H\in\mathfrak t\oplus\mathfrak a\}$$
 is the root space of complex dimension $1$ with respect to $\alpha$.  By \cite{Hel}, one can choose $X_{\alpha}\in V_{\alpha}$ such that $X_{-\alpha}=-\iota(X_{\alpha})$ and
$[X_{\alpha},X_{-\alpha}]=\alpha^{\vee},$ where $\iota$ is the Cartan involution and $\alpha^{\vee}$ is the dual of $\alpha$ by the Killing form.
Let
\begin{align}\label{lie-vector}
E_{\alpha}=X_{\alpha}-X_{-\alpha},~E_{-\alpha}=J(X_{\alpha}+X_{-\alpha}).
\end{align}
  Denoted by $\mathfrak k_{\alpha},\,\mathfrak k_{-\alpha}$ the real line spanned by $E_\alpha,\,E_{-\alpha}$, respectively.
Then we get  the Cartan decomposition of  Lie algebra $\mathfrak k$ of $\mathbf K$ as follows,
$$\mathfrak k=\mathfrak t\oplus\left(\oplus_{\alpha\in\Phi_+}\left(\mathfrak k_{\alpha}\oplus\mathfrak k_{-\alpha}\right)\right).$$

 Choose a real basis $\{E^0_1,...,E^0_r\}$ of $\mathfrak t$, where  $r$ is the dimension of $\mathbf T$.  Then $\{E^0_1,...,E^0_r\}$ together with $\{E_{\alpha},E_{-\alpha}\}_{\alpha\in\Phi_+}$ forms a real basis of $\mathfrak k$, which is indexed by $\{E_1,...,E_n\}$. We can also regard $\{E_1,...,E_n\}$
 as a complex basis of $\mathfrak g$. For any $g\in \mathbf G$, we define local coordinates $\{z_{(g)}^i\}_{i=1,...,n}$ on a neighborhood of $g$ by
\begin{align}\label{local-coordi}(z_{(g)}^i)\to\exp(z_{(g)}^i E_i)g.
\end{align}
It is easy to see that $\theta^i|_g\,=\,dz_{(g)}^i|_g$,  where the dual $\theta^i$ of $E_i$  is a right-invariant holomorphic $1$-form.
Thus
\begin{align}\label{Haar-mes}
dV_{\mathbf G}|_g:=\displaystyle{\wedge_{i=1}^n\left(dz_{(g)}^i\wedge d\bar{z}_{(g)}^i\right)}|_g,~\forall g\in \mathbf G
\end{align}
is also a right-invariant $(n,n)$-form,
which defines a Haar measure.

For a $K\times K$-invariant function $\psi$, Delcroix computed the Hessian of $\psi$ in the above local coordinates as follows \cite[Theorem 1.2]{Del1}.

\begin{lem}\label{Hessian}
Let $\psi$ be a $\mathbf K\times \mathbf K$ invariant function on $\mathbf G$.
Then  for any $x\in \mathfrak{a}_+$,
the complex Hessian matrix of $\psi$  in the  above coordinates is diagonal by blocks, and equals to
\begin{equation}\label{hess}
\mathrm{Hess}_{\mathbb{C}}(\phi_{\psi})(\exp(x)) =
\begin{pmatrix}
\frac{1}{4}\mathrm{Hess}_{\mathbb{R}}(\psi)(x)& 0 &  & & 0 \\
 0 & M_{\alpha_{(1)}}(x) & & & 0 \\
 0 & 0 & \ddots & & \vdots \\
\vdots & \vdots & & \ddots & 0\\
 0 & 0 &  & & M_{\alpha_{(\frac{n-r}{2})}}(x)\\
\end{pmatrix},
\end{equation}
where $\Phi_+=\{\alpha_{(1)},...,\alpha_{(\frac{n-r}{2})}\}$ is the set of positive roots and
\[
M_{\alpha_{(i)}}(x) = \frac{1}{2}(\alpha_{(i)},\nabla \psi(x))
\begin{pmatrix}
\coth\alpha_{(i)}(x) & \sqrt{-1} \\
-\sqrt{-1} & \coth\alpha_{(i)}(x) \\
\end{pmatrix}.
\]
We denote $(,)$ the Killing inner product of dual space of Lie algebra $\mathfrak t$.
\end{lem}

By (\ref{hess}) in Lemma \ref{Hessian},
we see that $\psi$ induced by a  $\mathbf K\times \mathbf K$-invariant K\"ahler form $\omega$ is convex on $\mathfrak{a}$.
The complex  MA  measure is given by
\begin{align}\label{MA-omega}
\omega_\psi^n\,=\,(\sqrt{-1}\partial\bar{\partial}\phi_\psi)^n\,=\,\mathrm{MA}_{\mathbb C}(\phi_\psi)\,dV_{\mathbf G}.
\end{align}
By \eqref{Haar-mes}, for any $x\in\mathfrak a_+$ we have
\begin{equation}\label{MA}
\mathrm{MA}_{\mathbb{C}}(\phi_\psi)(\exp(x))\,=\, \frac{1}{2^{r+n}}
\mathrm{MA}_{\mathbb{R}}(\psi)(x)\frac{1}{J(x)}\prod_{\alpha \in \Phi_+}(\alpha, \nabla \psi(x))^2
\end{equation}
In particular, by Proposition \ref{KAK int},
\begin{align}\label{volume}
{\rm vol}(M, \omega_\psi)\,=\,\int_M \omega_\psi^n=C_0\int_{2P_+}  \pi\,dy\,=C_0V_0=\,V,
\end{align}
where
 $P_+$ is the quotient   space  of  moment polytope $P$ by  the Weyl group $W$.
 and
 \begin{align}\label{pi}
\pi(y)=\prod_{\alpha \in \Phi_+}(\alpha, y)^2.
\end{align}

For a $\mathbf K\times \mathbf K$-invariant KE metric defined by a convex function $\psi$ as in  Lemma \ref{Hessian},  by (\ref{MA}), it is easy to  reduce the KE equation  to the following  real MA-type  equation  on the cone  $\mathfrak a_+$  (cf. \cite{Del1}),
\begin{align}\label{reduced-KE-equ-g}
 \pi (\nabla\psi){\rm Hess}(\nabla^2 \psi)= J(x) e^{-\psi}.
 \end{align}

\subsection{Geometry of horosymmetry spaces}

 Let $\mathbf H$ be a subgroup of  $\mathbf G$.   A  homogeneous   space $\mathbf G/\mathbf H$ is called  horosymmetric if there is a parabolic subgroup  $\mathbf P$  of  $\mathbf G$ such that   $\mathbf P/\mathbf H$ is a
  symmetric  space and  $\mathbf G/\mathbf H$ is a    $\mathbf P/\mathbf H$-fibration  over the generalized flag  manifold $\mathbf G/\mathbf P$ \cite{Del3}.
  In particular,  if the Lie algebra  of $\mathbf H$  contains all $X_\alpha$ with $\alpha\in \Phi_+$,    the homogeneous   space $\mathbf G/\mathbf H$ is called   horospherical.

According to  \cite{Del3}, there is a Levi subgroup $\mathbf L$ of  $\mathbf P$ and an involution $\sigma$ of $\mathbf L$ such that the following is true for the Lie algebras of $\mathbf P$ and $\mathbf H$:

1) The Lie algebra of $\mathbf P$ can be divided  into three parts as follows,
\begin{align}\label{lie-p}\mathfrak p=\mathfrak t+\sum_{\alpha\in \Phi_{\mathbf L}}X_\alpha + \sum_{\beta\in \Phi_{u}} X_{-\beta},
\end{align}
where
$\Phi_L$  is a root system  of $\mathbf L $, which is a subset of  root system $\Phi$ of $\mathbf G$
 and $\Phi_{u}$ is a subset of positive root system $\Phi_+$ of $\mathbf G$.

2)
The Lie algebra of $\mathbf H$ can be represented  as,
\begin{align}\label{lie-h}\mathfrak h=\mathfrak t_0  +   \sum_{\alpha'\in \Phi_{{\mathbf L}^\sigma}} X_{\alpha'} + \sum_{\beta\in \Phi_{u}} X_{-\beta},
\end{align}
where  $\mathfrak t_0$ is a subtorus of  $\mathfrak t$ which is fixed by $\sigma$ and
 $\Phi_{{\mathbf L}^\sigma}$ is a root system  of fixed subgroup $\mathbf L^\sigma $ of $\mathbf L$.
 Clearly,  $\mathbf G/\mathbf H$ is  horospherical if and only if  $\Phi_{\mathbf L}=\Phi_{\mathbf L^\sigma} $.

On a  horosymmetric spaces $\mathbf G/\mathbf H$,      we can consider   $\mathbf K$-invariant metrics
   (cf. \cite {Del3}). Let
   $$\Phi_s^+=\Phi_+\setminus \Phi_{\mathbf L^\sigma}$$
and
$$\rho_u=\frac{1}{2}\sum_{\beta\in \Phi_u} \beta.$$
Then  for a $\mathbf K$-invariant metric, there is a convex function $\psi$ on   $\mathfrak a'$, where $\mathfrak a'$ is the real part of $\mathfrak t'$  $(\cong \mathfrak t/\mathfrak t_0)$ with dimension $r_0$ such that the metric can be  divided  into three  parts as follows (cf. \cite{Del3}),
\begin{align}\label{horosymmetry-metric}
\omega_{\psi}(\exp\{x\}\cdot \mathbf H) =&\frac{\sqrt{-1}}{4}\sum_{i,j=1,...,r_0} ({\rm Hess}(\psi))_{i\bar j} dz^i\wedge d\bar z^j\notag\\
&+\frac{\sqrt{-1}}{2}\sum_{\alpha\in \Phi_s^+}   \frac {(\alpha, \nabla\psi)}{\sinh<\alpha,\cdot>} dz^{i_{\alpha}}\wedge d\bar z^{i_{\alpha}}\notag\\
&+\frac{\sqrt{-1}}{2}\sum_{\beta\in \Phi_u}   e^{-<\beta, x>} (\beta,\nabla\psi) dz^{i_\beta}\wedge d\bar z^{i_\beta}.
\end{align}
Since $\psi$ is invariant under the restricted Weyl group associated to    $\Phi_s^+$ \cite{Del3},  $x\in \mathfrak a'$ can be restricted on a  cone $\mathfrak a_+'$ in  $\mathfrak a'$,
\begin{align}\label{cone-a'}\mathfrak a_+'=\{x\in \mathfrak a'|~<\alpha, x>>0,~\forall ~\alpha\in\Phi_s^+\}.
\end{align}

\begin{ex}\label{ex-g}  A reductive Lie group $\mathbf G$ can be regarded as a  horosymmetry space
$${\mathbf G}=\hat{\mathbf G}/{\rm diag}(\mathbf G\times \mathbf G)$$
with the parabolic subgroup   $\hat{\mathbf P}=\hat {\mathbf G}=\mathbf G\times \mathbf G$. By  $\sigma(g_1, g_2)=(g_2, g_1)$, we have
$$\hat{\Phi}_s^+= \{(\alpha, 0)\cup(0,-\alpha')|~\alpha, \alpha'\in  \Phi_+\}.$$

\end{ex}

By (\ref{horosymmetry-metric}), we have
\begin{align}\label{horosymmetry-G}
\omega_{\psi}(\exp\{x\}\cdot \mathbf H) &=\frac{\sqrt{-1}}{4}\sum_{i,j=1,...,r} ({\rm Hess}(\psi))_{i\bar j} dz^i\wedge d\bar z^j\notag\\
&+\frac{\sqrt{-1}}{2}\sum_{\alpha\in \Phi_+}   \frac{(\alpha, \nabla\psi)}{\sinh<\alpha,\cdot>} (dz^{i_{\alpha}}\wedge d\bar z^{i_{\alpha}}+dz^{i_{-\alpha}}\wedge d\bar z^{i_{-\alpha}}).
\end{align}
Then by the relation (\ref{lie-vector}), it is easy to see that (\ref{horosymmetry-G}) coincides   with (\ref{hess}).

\begin{ex}\label{ex-degereration} $\mathbf C^*$-degenerations on   a $\mathbf G$-manifold.  Let $Z, Z'\in \mathfrak M$ on a horosymmetric  variety   $M$ and an integer $m$.  In \cite[Proposition 3.23]{Del2},  a pair  $(Y=Z-Z', m)$ can define  a  deformation $\mathcal M$ of  $\mathbf C^*$-degeneration on $M$  which  can be regarded as a compactification of homogenous space
  $$\mathbf G \times \mathbb C^*/ \mathbf H\times \mathbf e. $$
 In case of  $\mathbf G$-manifolds,  $(Z, Z')$ induces a restricted  $\mathbf C^*$-degeneration on  $\mathbf G$ as a  horosymmetry space in Example \ref{ex-g} has been described in \cite{LL}  as follows.

Let $\{\alpha_1,..., \alpha_{r_0}\}$ be  a subset  of $\Phi_+$  which consists of all simple roots in $\Phi_+$ such that
\begin{align}\label{tangent}<\alpha_i, Y>=0,~i=1,...,r_0.
\end{align}
Denote
 $\{\alpha_{r_0+1},..., \alpha_{r_1}\}$ to  be the set of  other simple roots in $\Phi_+$ which does  not satisfy   (\ref{tangent}).
Then there is a subset of  $\Phi_+$,  $ \{\alpha_{r_1+1},..., \alpha_{r_2}\}$  each of  which  is a combinator in  $\{\alpha_1,..., \alpha_{r_0}\}$.
Denote    $ \{\alpha_{r_2+1},..., \alpha_{r_3}\}$ to be the  remaining  subset of  $\Phi_+$ each of  which does not lie in  the above three subsets of  $\Phi_+$.

We define a subgroup $\mathbf H$ of $\hat{\mathbf G}=\mathbf G\times \mathbf G$ with  its Lie  subalgebra  $\mathfrak h$  generated by a basis
\begin{align}\label{lie-subalgebra}
&\{(Y, Y),  {\rm diag}\{Y^{\bot}\};\notag\\
&(X_{\alpha}, X_{\alpha}), (X_{-\alpha}, X_{-\alpha}), ~\alpha\in   \Phi_+\setminus \Phi_u=\{\alpha_1,..., \alpha_{r_0}, \alpha_{r_1+1},..., \alpha_{r_2}\};\notag\\
&(X_{\beta}, 0), (0, X_{-\beta}), ~\beta\in  \Phi_u= \{\alpha_{r_0+1},..., \alpha_{r_1}, \alpha_{r_2+1},..., \alpha_{r_3}\} \}.
\end{align}
Then $\mathbf N=\hat{\mathbf G}/\mathbf H$ is a horosymmetric space. The  parabolic Lie subgroup $\mathbf P$ of $ \hat{\mathbf G}$  is generated by a basis
 \begin{align}\label{parabolic-subg-1} &\{(E_i, E_j), ~i, j=1,...,r;  \notag\\
 &(X_{\alpha}, X_{\alpha'}), (X_{-\alpha},  X_{-\alpha'}), ~\alpha, \alpha'\in  \Phi_+\setminus \Phi_u; \notag\\
&(X_{\beta}, 0), (0,  X_{-\beta}), ~\beta\in \Phi_u\}.
 \end{align}

\end{ex}

In Example \ref{ex-degereration},  we have
$$\hat{\Phi}_s^+=\{(\alpha,0)\cup (0, -\alpha) |~\alpha \in \Phi_+\setminus \Phi_u\}$$
and
$$\hat\Phi_u=\{(\beta,0)\cup (0, -\beta)|~\beta\in \Phi_u \}, $$
where   $ \Phi_u=\{\alpha_{r_0+1},..., \alpha_{r_1}, \alpha_{r_2+1},..., \alpha_{r_3}\}$.
Thus any $\mathbf K\times \mathbf K$-invariant metric on  $\mathbf N$ is  determined  by  a convex function $\psi$ such that
\begin{align}\label{horosymmetry-LL}
\omega_{\psi}(\exp\{x\}\cdot \mathbf H) &=\frac{\sqrt{-1}}{4}\sum_{i,j=1,...,r}  ({\rm Hess}(\psi))_{i\bar j} dz^i\wedge d\bar z^j\notag\\
&+\frac{\sqrt{-1}}{2}\sum_{\alpha\in \Phi_+\setminus \Phi_u}   \frac{(\alpha, \nabla\psi)}{\sinh<\alpha,\cdot>} (dz^{i_{\alpha}}\wedge d\bar z^{i_{\alpha}}+dz^{i_{-\alpha}}\wedge d\bar z^{i_{-\alpha}})\notag\\
&+\frac{\sqrt{-1}}{2}\sum_{\beta\in \Phi_u}   e^{-<\beta, x>}  (\beta, \nabla\psi ) (dz^{i_\beta}\wedge d\bar z^{i_\beta} +dz^{i_{-\beta}}\wedge d\bar z^{i_{-\beta}}).
\end{align}
The above $x\in \mathfrak a$  can be  restricted on a cone $ \mathfrak a_+'$
in   $\mathfrak a$  associated to  the root subsystem $\hat\Phi_s^{+}$,
$$\mathfrak a_+'=\{x\in \mathfrak a|~<\alpha, x>>0,~\forall ~\alpha\in\Phi_+\setminus \Phi_u \}. $$

\subsection{Reduced KR soliton equation  on horosymmetry manifolds }

 As in (\ref{j-function}),   we introduce a function on   $ \mathfrak a_+'$ by
\begin{align}\label{j'-function} J_0(x)\,=\,\prod_{\alpha \in \Phi_s^+} \sinh\alpha(x).
\end{align}
Then, analogous to (\ref{reduced-KE-equ-g}), by (\ref{horosymmetry-metric}),   the KE   equation on a  horosymmetry manifold  can be reduced to the following  real MA-type equation on $\mathfrak a_+'$,
\begin{align}\label{reduced-KE-equ}
 \pi_0 (\nabla\psi+ 2\rho_u){\rm Hess}(\nabla^2 \psi)= J_0(x) e^{-\psi},
 \end{align}
 where
 \begin{align}\label{pi'}
\pi_0(y)= \prod_{\beta \in \Phi_s^+\cup \Phi_u}(\alpha, y).
\end{align}
It is known that  for any convex function $\psi$  in  (\ref{horosymmetry-metric})   there is an $A_0>0$ such that  (cf. \cite{PS}),
\begin{align}\label{unipotent-part}
\prod_{\beta \in \Phi_u}(\nabla\psi(x)+ 2\rho_u, \beta)\ge A_0, ~x\in ~ \mathfrak a_+.
\end{align}

Similarly,  for a KR soliton with respect to a HVF $X$ induced by an element in the  Lie algebra $\mathfrak t'$, we get a reduced KR soliton equation on $\mathfrak a_+'$ (also see \cite{DH}),
\begin{align}\label{reduced-KR-soliton}
 \pi_0(\nabla\psi+ 2\rho_u){\rm Hess}(\nabla^2 \psi)= J_0(x) e^{-\psi-X(\psi)}.
 \end{align}
 Thus in Example \ref{ex-degereration},  by (\ref{horosymmetry-LL}),  (\ref{reduced-KR-soliton}) becomes  \footnote{Here $2\rho$ in (\ref{reduced-KR-soliton})  is replaced by $4\rho$ by the notation $\Phi_u$ in (\ref{lie-subalgebra}).}

    \begin{align}\label{KR-soliton-LL}
 \pi(\nabla\psi+ 4\rho_u){\rm Hess}(\nabla^2 \psi)= J'(x) e^{-\psi-X(\psi)},
  \end{align}
 where $\pi(y)$ is defined by (\ref{pi}) and
 \begin{align}\label{j'-function-n} J'(x)\,=\,\prod_{\alpha \in \Phi_+\setminus \Phi_u} \sinh^2\alpha(x).
\end{align}
(\ref{KR-soliton-LL}) is defined on a   cone $\mathfrak a_+'$
in   $\mathfrak a$  by
$$\mathfrak a_+'=\{x\in \mathfrak a|~<\alpha, x>>0,~\forall ~\alpha\in \Phi_{+}\setminus \Phi_u\}. $$

\section{Reduced K\"ahler-Ricci flow}

By (\ref{reduced-KE-equ-g}), for a  $\mathbf K\times \mathbf K$-invariant initial metric in $2\pi c_1(M)$,  KR flow (\ref{KRF})  on a Fano $\mathbf G$-manifold
 $M$ can be reduced to  the following parabolic equation of MA type   on $\mathfrak a_+$,
  \begin{align}\label{reduced-KRflow}
\frac{\partial \psi}{\partial t}=\log [ \pi (\nabla\psi){\rm Hess}(\nabla^2 \psi)]+(\psi+j),
   \end{align}
where
  $$j(x)=-\log J(x), ~x\in  \mathfrak{a}_+.$$
The solution $\psi=\psi_t$ is a family of   Weyl-invariant functions on ${\mathfrak a}$ and  $-\frac{\partial \psi}{\partial t}$ is a Ricci potential of metric $\omega_\psi$. By a result of  Perelman (cf. \cite{ST, TZ2, WfZ1}), there is a  constant $c_t$ so that
$$h_t=-\frac{\partial \psi}{\partial t}+c_t$$
 is uniformly bounded. In fact, such $c_t$ can be normalized by   the identity
 \begin{align}\label{normalization}\int_{\mathfrak a_+}e^{h_t}\pi (\nabla\psi_t){\rm Hess}(\nabla^2 \psi_t)dx= \int_{2P_+}  \pi\,dy=V_0.
 \end{align}

For simplicity,  we let $w_t=\psi_t+j$  as a convex function on $\mathfrak a_+$.
  Then  (\ref{reduced-KRflow})  can be rewritten   as a family of MA  type equations
\begin{align}\label{elliptic-family}
 \pi (\nabla\psi_t){\rm Hess}(\nabla^2 \psi_t) = e^{-h_t}e^{-w_t},
  \end{align}
  Note that $j(x)\to\infty$ as $x\to W_\alpha$,  where $W_\alpha=\{x\in \mathfrak a|~<\alpha, x>=0\}$ is a Weyl well associated to positive root $\alpha$.  Thus
there are an $x_t \in \mathfrak a_+$ and a number  $m_t$ such that
\begin{align}\label{m-t}
 \inf_{\mathfrak{a}_+} w_t= w_t(x_t)=m_t.
 \end{align}

As in \cite{Del1},  for any nonnegative integer $k$, we set
$$A_k=\{x\in \mathfrak a_+: ~\ m_t+k\le w_t(x)\le m_t+k+1\}. $$
Then each $U_k=\bigcup_{i=0}^k A_i=\{w_t<m_t+k+1\}$ is a
convex set.
Moreover,  for any $k\ge 0$, $A_k$ is a
bounded set and the minimum $m_t$ is attained at some point in
$A_0$.

Since
$$\nabla (-j)(x_t)= \nabla \psi_t(x_t)\in P_+,$$
$\nabla (-j)(x_t)$ is uniformly bounded. Thus  there is a uniform constant $\delta_0$ such that
\begin{align}\label{xi-condition-1}
 <\alpha, x_t>\ge \delta_0, ~\forall \alpha\in \Phi_+.
\end{align}
This implies that there is a ball $B_{\kappa}(x_t)\subset A_0$ with a fixed radius such that
\begin{align}\label{xi-condition-1-1}
 {\rm dist}(B{\kappa}(x_t), W_\alpha)\ge \delta_0', ~\forall \alpha \in \Phi_+.
\end{align}

Due to \cite{Del1}, we  also have the following estimates for  the equation (\ref{elliptic-family}).

\begin{prop}\label{concentration} The following estimates hold: 1)
\begin{align}\label{xi-condition-2}
  \inf_{\mathfrak{a}_+} (\psi_t+j)= (\psi+j)(x_t)=m_t,  ~|m_t|\le C;
\end{align}

2)  ${\rm diam}(U_k)\le D_k$.
\end{prop}

\begin{proof}A version of  estimate  (\ref{xi-condition-2}) for $m_t$  has been obtained for a family of MA  type equations raised for the existence problem  of KE metrics via the continuity method on toric manifolds and $\mathbf G$-manifolds  \cite[Lemma 3.2]{WxZ}, \cite[Proposition 4.4]{Del1}, respectively. Thus following the argument for the KR flow on  toric manifolds \cite{Zhu1},  we can also get  the estimate  (\ref{xi-condition-2}). We leave the detailed proof to the reader.

The estimate  2) follows from  (\ref{xi-condition-2}) and  the equation (\ref{elliptic-family}) together with  the  normalization  (\ref{normalization}) via  the volume.

\end{proof}

To prove Theorem \ref{main-theorem},  we  establish  the following explicit  result.

\begin{theo}\label{main-theorem-2}Let $(M, J)$ be a Fano $\mathbf G$-manifold.   Then there are only three cases for the KR flow (\ref{reduced-KRflow}) as follows.

Case 1).  There is a sequence of $t_i$ such that
 $$|x_{t_i}|\le C.$$
 Then (\ref{reduced-KRflow}) converges to a  KE metric in sense of K\"ahler potentials. As a consequence, $(M,J)$
  is a KE-manifold.

Case 2). $|x_t|\to  \infty$ as $t\to \infty$ and there is a sequence of $t_i$ such that
 $$|Pj(x_{t_i})|\le C,$$
 where $Pj$ is the projection from $(\mathfrak a_c,  \mathfrak a)\to  \mathfrak  a$ with the center  $\mathfrak a_c$ of
 $\mathfrak g$.
 Then (\ref{reduced-KRflow}) converges to a  KR soliton in sense of K\"ahler potentials. As a consequence, $(M,J)$  admits a KR soliton.

Case 3).   $|Pj(x_t)|\to  \infty$ as $t\to \infty$. Then there are two subcases:

Case 3.1).
 For any $\alpha\in \Phi_+$, it holds
\begin{align}\label{general-points}<W_\alpha, x_{t}>\to\infty, ~{\rm as}~t\to \infty.
\end{align}
 Then (\ref{reduced-KRflow}) converges  locally smoothly to a KR soliton $\omega_{KS}$  on a horospherical space in Cheeger-Gromov topology,  whose completion is the  GH  limit of KR flow  (\ref{KRF}) with  a structure of  $\mathbb Q$-Fano horospherical   variety $ M_\infty$  as a limit of  $\mathbb C^*$-degeneration of $(M,J)$  induced by an element in the Lie algebra of Cartan torus.   Moreover, $ M_\infty$ is same with    a  limit of  $\mathbb C^*$-degeneration of $(M,J)$ induced by the soliton HVF of  $\omega_{KS}$,  if  $\omega_{KS}$ is not a KE metric.

Case 3.2).   (\ref{general-points}) does not hold.  Then there is  a subset  $\Phi_u$ of $\Phi_+$
such that
\begin{align}\label{case-1-1}<\beta, x_{t}>\to \infty ~{\rm as} ~t\to\infty, ~\forall \beta\in   \Phi_u
\end{align}
and
\begin{align}\label{case-2-1}\delta_0\le<\alpha', x_{t}>\le A, ~\forall \alpha'\in \Phi_+\setminus  \Phi_u.
\end{align}
Moreover,   (\ref{reduced-KRflow}) converges   locally smoothly to a KR soliton  $\omega_{KS}$  on a horosymmetric space in Cheeger-Gromov topology as in Case 3.1):   whose completion is the  GH  limit of KR flow  (\ref{KRF}) with  a structure of  $\mathbb Q$-Fano horosymmetric  variety  $ M_\infty$  as a limit of  $\mathbb C^*$-degeneration  of $(M,J)$  induced by an element in the Lie algebra of Cartan torus;    in  case that  $\omega_{KS}$ is not a KE metric,  $ M_\infty$ is  same with  a  limit of  $\mathbb C^*$-degeneration of $(M,J)$ induced by the soliton HVF of  $\omega_{KS}$.

\end{theo}

\begin{rem}\label {theorem2-remark}  1) If $\mathbf G$ is semi-simple, Case 2) will not happen in Theorem \ref{main-theorem-2}.

2) In Case 2), when $(M,J)$ is a toric manifold, the result was proved in \cite[Theorem 1.1]{Zhu1}.

3) Case 3.1) is just Case 3.2) when $\Phi_u=\Phi_+$. The lower bound of (\ref{case-2-1}) comes from (\ref{xi-condition-1}).

4) In  Case 3), $\Phi_u$ will not be empty and so  the curvature along the flow must blow-up  according to the proof of \cite[Lemma 4.4]{LTZ1} (also see \cite[Lemma 6.4]{Zhu3}).  \footnote{In fact,  the flow must blow-up  if there is a $\beta\in   \Phi_+$ and a sequence $\{x_{t_i}\}$ such that  (\ref{case-1-1}) holds.}  In particular,   $(M, J)$ could not admit
 any KR soliton \cite{TZZZ, DS}.  Thus together with the results in  case 1) and 2),    Theorem \ref{main-theorem-2} also implies Theorem \ref{g-flow}.

\end{rem}

Theorem \ref{main-theorem} will be proved in Section 5.  In the rest  of this section,  we prove several fundamental lemmas for  Case 3.2) in the theorem which will be used in next two sections.

\begin{lem}\label{general-case} Suppose that   $|Pj(x_t)|\to  \infty$ as $t\to \infty$  in Case 3.2).   Let
\begin{align}\label{x-sequence}X_t=\frac{x_t}{|x_t|}.
\end{align}
Then  there is a subset  $\Phi_u$ of $\Phi_+$   such that
\begin{align}\label{orthogonal}
&\lim_t<\nabla\log{\rm sinh}(<\beta, x>)(x_{t_i}), X>=<\beta, X>, ~\forall \beta\in   \Phi_u,~X\in \mathfrak a; \notag\\
&<\alpha', X_\infty>=0, ~\forall \alpha' \in \Phi_+\setminus  \Phi_u,
\end{align}
where
$X_\infty$ is   a limit of  any sequence of $X_{t}$.

\end{lem}

\begin{proof} By  (\ref{xi-condition-1}), it is easy to see that
there are   a subset  $\Phi_u$ of $\Phi_+$ and  a  sequence of  $t_i$ such that (\ref{case-1-1}) and (\ref{case-2-1}) are both satisfied.
 Since the subset  $\Phi_u$ of $\Phi_+$
 can be  finitely possible chosen,  it is easy  to see that   both of  (\ref{case-1-1}) and (\ref{case-2-1}) hold uniformly at $t$.
In  the case of (\ref{case-1-1}),    for any vector $X\in \mathfrak a$, we have
$$\lim_t <\nabla\ln{\rm sinh}(<\beta,x>)(x_{t}), X>=<\beta,X>, ~\forall \beta\in   \Phi_u.$$
In the case of  (\ref{case-2-1}),  we get
$$<\beta, X_\infty>=\lim_t  \frac{1}{|x_{t}|}<\alpha, x_i>=0,$$
since $|x_t|\to\infty$.

\end{proof}

 \begin{lem}\label{regular-uk}Under the assumption in  Lemma \ref{general-case}, we have
\begin{align}\label{delta-t-distance}
 {\rm dist}(U_k, W_\alpha)\ge \delta_k>0 ~\forall \alpha \in \Phi_+.
\end{align}

 \end{lem}

\begin{proof}   By   2) in Proposition \ref{concentration},  we see that
$$|\psi(x)-\psi(x_t)|\le C_1,  ~\forall~x\in U_k.$$
Thus, it suffices to verify  (\ref{delta-t-distance}) for $\alpha \in \Phi_+\setminus  \Phi_u.$
In fact, by (\ref{case-1-1}) and (\ref{case-2-1}),  we have
\begin{align} &|w_k(x)-w_k(x_{t})|\ge -C_1+|j(x)-j(x_{t})|\notag\\
&\ge -C_2+|- \prod_{\alpha\in\Phi_+\setminus  \Phi_u}  \ln{\rm sinh}(<\alpha, x> +  \prod_{\alpha\in\Phi_+\setminus  \Phi_u} \ln{\rm sinh}(<\alpha,x_{t}>|\notag\\
&\ge -C_3+|- \prod_{\alpha\in\Phi_+\setminus  \Phi_u}  \ln{\rm sinh}(<\alpha, x>|.\notag
\end{align}
Then
$$|- \prod_{\alpha\in\Phi_+\setminus  \Phi_u}  \ln{\rm sinh}(<\alpha, x>|\le C_k$$
and so
$$<\alpha, x>\ge \epsilon_k, ~\forall \alpha \in \Phi_+\setminus  \Phi_u, ~x\in U_k.$$
  This implies (\ref{delta-t-distance}).

\end{proof}

\begin{lem}\label{limit-mt}
Let
 $$\rho_u=\frac{1}{2}\sum_{\beta\in \Phi_u} \beta.$$
Then
\begin{align}\label{psi-norm}|\psi_{t}(x_t)-<4\rho_u,x_{t}>|
\le C.
\end{align}
\end{lem}

\begin{proof}By  (\ref{case-1-1}) and  (\ref{case-2-1}), it is easy to see that
$$|j(x_{t})+<4\rho_u,x_{t}>|\le C.$$
Then by (\ref{xi-condition-2}), we get (\ref{psi-norm}).

\end{proof}

\begin{lem}\label{limit-J} Let
$$\hat J'(x)=\lim_{t}\prod_ {\alpha'\in \Phi_+\setminus\Phi_u} {\rm sinh}^2(<\alpha', x+x_{t}>).$$
Then  there is a vector $\mathbf a_0\in \mathfrak a_+$ such  that
\begin{align}\label{a-existence}\lim_t {\rm sinh}(<\alpha', x_{t}>)  = {\rm sinh}(<\alpha', \mathbf a_0>), ~\forall ~\alpha'\in \Phi_+\setminus\Phi_u.
\end{align}
As a consequence,
\begin{align}\label{limit-j}\hat J'(x)=\prod_ {\alpha'\in \Phi_+\setminus\Phi_u} {\rm sinh}^2(<\alpha', x+\mathbf a_0>).
\end{align}

\end{lem}

\begin{proof} Choose    a set of simple roots in $\Phi_+\setminus \Phi_u$ by $\{\alpha_{1},..., \alpha_{r_0}\}$.  Then by (\ref{case-2-1}),  there are numbers $\delta_{\alpha'}$ such that
$$\lim_i {\rm sinh}(<\alpha', x_{t}>)  =\delta_{\alpha'}, ~\forall~ \alpha'\in \{\alpha_{1},..., \alpha_{r_0}\}.$$
Let
$M_{\alpha'}$  be a hyperplane by
$$M_{\alpha'}=\{x\in \mathfrak a|~<\alpha', x>=\hat \delta_{\alpha'}\}.$$
Thus there is a vector  in  $\mathbf a_0\in \cap_{\alpha'\in \{\alpha_{1},..., \alpha_{r_0}\}}M_{\alpha'}$ such that
$$ {\rm sinh}(<\alpha', \mathbf a_0>)  =\delta_{\alpha'}, ~\forall~ \alpha'\in \{\alpha_{1},..., \alpha_{r_0}\}.$$
This implies that 
$$ {\rm sinh}(<\alpha', \mathbf a_0>)  =\delta_{\alpha'}, ~\forall~ \alpha' \in \Phi_+\setminus\Phi_u.
$$
and so (\ref{a-existence}) and  (\ref{limit-j}) hold.

\end{proof}

\section{Local $C^2$-estimate}

In this section,  we first give  an upper bound estimate  for  Hessian  of $\psi$ through the norm  estimate of  torus vector fields on $M$.   We prove the following  general result.

\begin{lem}\label{gradient-vector}Let $M$ be a Fano manifold and $\omega_g\in
2\pi c_1(M)$ be a K\"ahler metric. Suppose that
\begin{align}\label{assumption-geom} & |R_g|\le A_0, ~C_s (\omega_g)\le \Lambda_0,\notag\\
&{\rm vol}(B_r(x, \omega_g))\ge c_0r^{2n},~\forall ~r<1,~x\in M,
\end{align}
where $R_g$ is the scalar Ricci curvature of $\omega_g$, $ C_s(\omega_g)$ is the Sobolev constant of $\omega_g$ and $A_0, \Lambda_0, c_0$ are uniform positive constants.
Then for any Hamitonian  HVF $X$ with  real-valued potential $\theta_X$ it holds
\begin{align}\label{f-grad-infty} |\nabla\theta_X|_g= |X|_{g}\le C(A_0, C_s, c_0). \end{align}.

\end{lem}

\begin{proof}
Note that the diameter of $(M,\omega_g)$ is uniformly bounded by the last condition in (\ref{assumption-geom}) and the fact $\omega_g\in 2\pi c_1(M)$. Then by a result of Cheng, the Green function is uniformly bounded below.  Since
the Ricci potential $h$ of $\omega_g$ satisfies
\begin{align}\label{h-equation}\Delta_g h=R_g -n,
\end{align}
by the normalization
\begin{align} \int_M {h}\omega_g^n=0, \notag
\end{align}
and the Green formula, we see that
$h$ is uniformly bounded.
On the other hand, by (\ref{h-equation}), it is easy to see that
$\int_M |\nabla {h}|^2\omega_g^n$ is uniformly bounded.
Thus, by the gradient estimate in \cite[Theorem 7.1]{TZZZ}, there is
a uniform constant $C_1$ such that
\begin{align}\label{grad-h}|\nabla h|_g\le C_1.
\end{align}

  By adding a constant,  $\theta_X$ satisfies the following  equation (cf. \cite{Fu, TZ2}),
\begin{align}\label{fi-equation}\Delta_g \theta_X +\theta_X +\langle\partial \theta_X, \partial h\rangle_g=0.
\end{align}
Note that $\theta_X$ is uniformly bounded (cf. \cite{Zhu4,  LTZ1}).
Let $F=e^{h} \theta_X$. Then by (\ref{fi-equation}), we have
\begin{align}\Delta_g F &= -F + \theta_X \Delta_g e^{h}\\
&=-F+ e^{h}\theta_X(R_g-n+|\nabla h|_{g}^2).\notag
\end{align}
Since $e^{h}\theta_X(R-n+|\nabla h|_{g}^2)$ is uniformly bounded by (\ref{grad-h}),
$\int_M |\nabla {F}|^2\omega_g^n$ is uniformly bounded.
Thus, as in the proof of (\ref{grad-h}), we can apply \cite[Theorem 7.1]{TZZZ}  again to get the  gradient estimate,
$$|\nabla F|_{g}\le C.$$
As a consequence, $ |\nabla \theta_X|_g$ is uniformly bounded by (\ref{grad-h}).
The lemma is proved.

\end{proof}

We apply Lemma \ref{gradient-vector} to the metric $\omega_{\psi_t}\,=\,\sqrt{-1}\partial\bar\partial \phi_{\psi_t}$ in RKR flow (\ref{reduced-KRflow}) on a  Fano $G$-manifold.  Note that by the  Perelman's estimate  (cf. \cite{ST, TZ3, WfZ1}),  the conditions in  (\ref{assumption-geom}) are all satisfied for $\omega_{\psi_t}$. Thus Lemma \ref{gradient-vector} holds for these K\"ahler metrics.

Let $\{e_1, ..., e_r\}$ be $r$ vectors fields generalized by  the  basis $\{E^0_1,...,E^0_r\}$ of $\mathfrak t$. Then  on $T^{\mathbb C}$,
$$e_A=\frac{\partial}{\partial x_A}+\sqrt{-1} J\frac{\partial}{\partial x_A}=\frac{\partial}{
\partial z^A}, ~A=1,...,r.$$
It follows that
\begin{align}\label{e-norm}|e_A|_{\omega_{\psi_t}}^2= \omega_{\psi_t}(\frac{\partial}{
\partial z^A},\overline{\frac{\partial}{
 \partial z^A}})=\psi_{t,AA},  ~A=1,...,r,
 \end{align}

\begin{prop}\label{upper-bound-Hessian}

$${\rm Hess}(\nabla^2 \psi_t) \le C.$$

\end{prop}

\begin{proof} Let
$$X=\sum_A  a_A e_A$$
be some real numbers $a_A$, $A=1,...,r$.
Then by Lemma \ref{gradient-vector},
$$|X|_{\omega_{\psi_t}}^2\le C.$$
Namely by (\ref{e-norm}),
$$\sum a_Aa_B\psi_{t,AB}\le C.$$
Thus the proposition is true since $a_A$ can be chosen arbitrary.

\end{proof}

 \begin{prop}\label{local-uniform}Let
 $$\hat\psi_t=\psi_t(\cdot+x_t)-\psi_t(x_t).$$
 Then $\hat\psi_t$ is uniformly $C^{l,\alpha}$-bounded  on each $U_k$.

\end{prop}

\begin{proof}
 Since $<\nabla\psi_t, \alpha>$ is uniformly bounded,
by (\ref{elliptic-family}) and 1) in Proposition  \ref{concentration}, we have
 \begin{align}\label{lower-bound-hessian-1}
 {\rm det}(\nabla^2 \psi_t)  \ge  B_k, ~{\rm on}~ U_k
 \end{align}
 Then by Proposition \ref{upper-bound-Hessian}, we get
$$B_k'\le {\rm Hess}(\nabla^2 \psi^i) \le C, ~{\rm on}~ U_k$$
Thus, by the regularity of  parabolic MA equation (\ref{reduced-KRflow})
   together with 2) in Proposition \ref{concentration} and Lemma  \ref{regular-uk},  we prove that
 $\psi_t(\cdot)$ is locally uniformly $C^{l,\alpha}$-bounded.
\end{proof}

 \begin{lem}\label{gradient-nondegenerate}On each $U_k$ it holds
\begin{align}\label{nabla-psi}<\nabla\psi_t, \alpha>\ge  C_k.
\end{align}

 \end{lem}

\begin{proof}
  By  Proposition \ref{upper-bound-Hessian} and (\ref{elliptic-family}),  we have
 $$\pi (\nabla\psi_t )\ge C_k', ~{\rm on}~ U_k.$$
 Note that $<\nabla\psi_t, \alpha>$ is uniformly bounded. Thus we also get (\ref{nabla-psi}).

\end{proof}

 \begin{prop}\label{limit-cone} Under the assumption in  Lemma \ref{general-case}, the set $U_k-x_t$ converges to a cone $\mathfrak a_+'\supset \mathfrak a_+$   in   $\mathfrak a$  as $t\to\infty$  by
 $$ \mathfrak a_+'=\{x\in \mathfrak a|~<W_{\alpha'}. x>>0, ~\forall ~\alpha'\in \Phi_+\setminus\Phi_u\}. $$
 \end{prop}

\begin{proof}
By Lemma \ref{limit-mt} and Lemma \ref{limit-J},  we have
\begin{align}
\label{w-asymptotic}
\hat w_t(x)=w_t(x+x_t)=\psi_t(x+x_t)-\psi_t(x_t)-<\rho_u, x> -\log \hat J'(x+\mathbf a_0)+O(1).
\end{align}
Let $B_r$ be $r$-ball centered at the original  in  $\mathfrak a$.  Then  the leading term of  $\hat w_t(x)$  in    $B_r\cap\mathfrak a_+'$ is of  $-\log \hat J'(x+\mathbf a_0)$. Thus the set   $\{x\in \mathfrak a|~\hat w_t(x)\le  k\}$ will exhaust  $\mathfrak a_+'$. Thus the proposition is true.
\end{proof}

\section{Proof of Theorem \ref{main-theorem-2}}

In this section, we prove  Theorem \ref{main-theorem-2} and so get  Theorem \ref{main-theorem}.   We first show  the existence of limit solution of RKR flow  (\ref{reduced-KRflow}).

\subsection{ Smooth KR soliton solution on $\mathbf N$}

 \begin{prop}\label{limit-solution} Suppose that  $x_t$  satisfies  (\ref{case-1-1}) and  (\ref{case-2-1}) in Theorem \ref{main-theorem-2}.  Let
$\hat J'(x)$ be given as  in (\ref{limit-j}) and
$$J'(x)=\hat J'(x-\mathbf a_0).$$
Then there is a  smooth function  $\phi$ on  $\mathfrak a_+'$ which  satisfies the following equation,
\begin{align}\label{limit-equation} \pi (\nabla\phi + 4\rho_u){\rm Hess}(\nabla^2 \phi)= J'(x)  e^{-\phi-Y(\phi)},
\end{align}
where $Y(\phi)=<Y, \nabla \phi>$ for some   $Y\in \mathfrak a_+'$ which satisfies
 \begin{align}\label{Y-vectors}
<\alpha', Y>=0, ~\forall~ \alpha' \in \Phi_+\setminus  \Phi_u.
\end{align}
\end{prop}

\begin{proof}
By Proposition \ref{local-uniform},  $\hat\psi_t$ is uniformly $C^{l,\alpha}$-bounded  on each $U_k-x_t$.
 Then  there is  a sequence of  $\hat\psi_{t_i}$  which  converges  locally uniformly  to a smooth function  $\psi_\infty$  on   $\mathfrak a_+'$.
On the other hand,  by  (\ref{elliptic-family}), we have
\begin{align}\label{elliptic-family-2}
& \pi (\nabla\hat\psi_{t_i} ){\rm Hess}(\nabla^2 \hat\psi_{t_i}) \notag\\
 &= e^{-h_{t_i}}   \prod_ {\alpha'\in \Phi_+\setminus\Phi_u} {\rm sinh}^2(<\alpha', x+x_{t_i}>)e^{-(\hat\psi_{t_i} +\psi_{t_i}(x_{t_i})- <4\rho_u, x_{t_i}>)}\notag\\
 & \times  e^{- <4\rho_u, x_{t_i}>}\prod_ {\beta\in \Phi_u} {\rm sinh}^2(<\beta,  x+x_{t_i}>).
  \end{align}
 By Lemma \ref{limit-mt} and Lemma \ref{limit-J},
 on  each  set $U_k-x_{t_i}$,   the function
  \begin{align}\label{function-1} e^{\psi_{t_i}(x_{t_i})- <4\rho_u, x_{t_i}>)} \times  e^{- <4\rho_u, x_{t_i}>}\prod_ {\beta\in \Phi_u} {\rm sinh}^2(<\beta, x_{t_i}+x>)\to A_0e^{<4\rho_u, x>},
\end{align}
  for some constant $A_0$,  and the function
 \begin{align}\label{function-2}  \prod_ {\alpha'\in \Phi_+\setminus\Phi_u} {\rm sinh}^2(<\alpha', x+x_{t_i}>)  \to \hat J'(x),
 \end{align}
 respectively. Moreover, the convergence are both uniformly  on a fixed domain $U\subset \mathfrak a_+'$ which is contained in $U_k-x_t$ as $k>>1$.

 \begin{lem}\label{x-vector}  Let
 $$\hat h_{t}(\cdot)= h_{t}(\cdot+x_{t}).$$
   Then there are  a $Y\in \mathfrak a_+$ and a  constant $c$ such that
    \begin{align}\label{ht-limit}
   \lim_{t} \hat h_{t}(\cdot) =<Y, \nabla \psi_\infty>+c.
   \end{align}
   The convergence of  $\hat h_{t}(\cdot)$  is locally uniformly   $C^{l,\alpha}$ as  same as $\hat\psi_{t}.$   Moreover,   if $x_t$  satisfies  (\ref{case-1-1}) and  (\ref{case-2-1}) in Theorem \ref{main-theorem-2},   $Y$ satisfies  (\ref{Y-vectors}).
 \end{lem}

By the above lemma together with (\ref{function-1}) and (\ref{function-2}),  we get from (\ref{elliptic-family-2}),
 \begin{align}\label{solution-psi'}\pi (\nabla\psi_\infty){\rm Hess}(\nabla^2 \psi_\infty)=A_0' J'(x-\mathbf a_0) e^{ <4\rho_u,  x>} e^{-\psi_\infty-Y(\psi_\infty)}.
 \end{align}
 for some constant $A_0'$.
    The above equation can be defined in   the  cone $\mathfrak a_+'$ by Proposition \ref{limit-cone}.
  Let
 \begin{align}\label{solution-phi}
 \phi(x)=\psi_\infty(x+\mathbf a_0) -<4\rho_u,  x> .
 \end{align}
 Then we obtain
 \begin{align}\label{solution-0}\pi (\nabla\phi +4\rho_u){\rm Hess}(\nabla^2 \phi)= \tilde A_0J'(x)  e^{-\phi-Y(\phi)}.
 \end{align}
 Hence,   by adding a constant to  $\phi(x)$,   $\phi(x)$ becomes a solution of (\ref{limit-equation}).

\end{proof}

\begin{proof}[Proof of Lemma \ref{x-vector}]  We will modify the argument in \cite[Lemma 4.1]{Zhu1} for toric  manifolds.  The difficulty at present is that $\hat\psi_t$ has only the local convergence at the present.

Step 1. Recall $x_t$ is the critical point of $w_t$ in (\ref{m-t}).   We claim:
there is an $\epsilon_0>0$ such that
   \begin{align}\label{x-distance} |x_{t}-x_{t+\epsilon_0}|\le C, ~\forall~t>0,
\end{align}
 where  $C$ is a   uniform constant.

 For any time $t>0$,    we consider  the  equation (\ref{reduced-KRflow}) on $[t, t+\epsilon_0]$,
\begin{align}\label{reduced-KRflow-2}
\frac{\partial \psi'}{\partial s}=\log [ \pi (\nabla\psi'){\rm Hess}(\nabla^2 \psi')]+(\psi'+j), ~s\in [t, t+\epsilon_0].
   \end{align}
By  the Perelman's result,   we may assume  that
$$\psi_{s}'=0, ~ |\frac{\partial \psi'}{\partial s}|_{t}=h_t|\le A,$$
where $h_t$ is chosen as in (\ref{elliptic-family}).
Then the solution $ \psi'_s$ is just different to a constant with $\psi_s$ in  (\ref{elliptic-family}).
By the maximum principle, we get
\begin{align}\label{small-psi}
|\psi_{t}'-\psi_{s}'|\le C_1\epsilon_0, ~|\frac{\partial \psi'}{\partial s}|\le 2A,~\forall ~s\in [t, t+\epsilon_0].
\end{align}

Let
$$w'=w_{s}'=\psi_{s}'+j.$$
Then   the minimal point of $w_{s}'$ is also same with $x_s$  of $w_{s}$  in (\ref{xi-condition-2}).
By Proposition \ref{local-uniform},  we see that
$$C_2^{-1} \le {\rm Hess}(\nabla^2 \psi_{t}') \le C_2, ~{\rm on}~ B_{\kappa}(x_{t}).
$$
It follows that
$$(C_2')^{-1} \le {\rm Hess}(\nabla^2 w_{t}') \le C_2', ~{\rm on}~ B_{\kappa}(x_{t}).
$$
Thus
  $$|w_t'(x_t)-w_t'(x)|\ge 4C_1\epsilon_0, ~\forall~x\in \partial  B_{\kappa}(x_{t}), $$
if  $\epsilon_0$ is chosen  small enough. As a consequence,
$$w_t'(x)\ge w_t'(x_t)+ 4C_1\epsilon_0, ~\forall~x\in \partial  B_{\kappa}(x_{t}).$$
Hence, by (\ref{small-psi}),  we get
  $$w_s'(x)\ge w_{s}'(x_t)+ 2C_1\epsilon_0\ge \inf_{ x'\in B_{\kappa}(x_{t})}w_s'(x')+ 2C_1\epsilon_0, ~\forall~x\in \partial  B_{\kappa}(x_{t}).$$
 Therefore,  there is  a critical point $x_{s}'$ of  $w_{s}'$  in $B_{\kappa}(x_{t})$ such that
  $$\nabla w_{s}'(x_{s}')=0.
  $$

By the convexity of $w_{s}'$, $x_{s}'$ is the  global minimal point of   $w_{s}'$.  Thus  $x_{s}'=x_s$.
 Moreover,
$$|x_{s}-x_{t}|\le \kappa, ~\forall~s\in [t, t+\epsilon_0].$$
In particular,
$$|x_{t+\epsilon}-x_{t}|\le \kappa, ~\forall~s\in [t, t+\epsilon_0].$$
Hence,  we prove (\ref{x-distance}).

By  (\ref{x-distance}),  we see that for any integer $l>0$ it holds
$$ |x_{l\epsilon_0}-x_{(l+1)\epsilon_0}|\le C.$$
  Then we  can choose a  family of
modified points $y_t$ in $\mathfrak a_+$ such that
\begin{align}\label{x-distance-3} |x_t-y_t|\le C~\text{and}~|\frac{dy_t}{dt}|\le C.
\end{align}
Note that the action
\begin{align}\label{holo-tran}x\to x+y_t
\end{align}
corresponds to a family of transformation by  the torus $T$. Thus
$\frac{dy_t}{dt}$  corresponds to a family of HVF  $e_t$ induced by  $\mathfrak  t$.
By (\ref{x-distance-3}),  $e_t$ converges to a limit  HVF $e_\infty$  induced by  $\mathfrak  t$.

Let $\tilde\psi'(\cdot)= \tilde\psi_t'(\cdot)=\psi_t'(y_t+\cdot)$ and
$$w_t'(x)=\tilde\psi_t'(x)+j(x+y_t).$$
Then
$$m_t'=\inf_{ \mathfrak a_+} w_t'(x)= w_t'(x_t)$$
also satisfies   1) in Proposition  \ref{concentration}.
Since $\nabla\psi'$ is uniformly bounded,
analogous to the function  in Proposition \ref{local-uniform},
$$\psi_t(x_t+\cdot)-\psi_t(x_t)=\psi_t'(x_t+\cdot)-\psi_t'(x_t),$$
\begin{align}\label{hat-psi}\hat\psi_t'=\tilde\psi_t'-\psi_t'(y_t)
\end{align}
is also uniformly $C^{l,\alpha}$-bounded  on each $U_k'=\{x\in \mathfrak a_+|~w_t'(x)-m_t'<k+1 \}$. As in Proposition \ref{limit-cone}.   $U_k'-x_t$ converges to the  cone $\mathfrak a_+'\supset \mathfrak a_+$.

Step 2.  We notice that
 \begin{align}\label{derative}\frac{\partial \psi'}{\partial s}=\frac{\partial \tilde\psi'}{\partial s}-e_s( \hat\psi'), ~\forall ~s\in  [l\epsilon_0, (l+1)\epsilon_0]
 \end{align}
 and $\frac{\partial \tilde\psi'}{\partial s}$ is convergent by (\ref{reduced-KRflow-2}) as well as $\hat\psi_s'$ is convergent.

 We claim:  on each $U_k'- x_t$ it holds
\begin{align}\label{h-holomorphic}
\lim_s \frac{\partial \hat\psi'}{\partial s}= {\rm const.}
\end{align}

On the contrary,  we suppose that (\ref{h-holomorphic}) is not true. Note that  $\frac{\partial \hat\psi'}{\partial s}$ is different to a constant with  $\frac{\partial \tilde\psi'}{\partial s}$. Then
$\frac{\partial \hat\psi'}{\partial s}$ converges to a smooth non-constant function on $\mathfrak a_+'$. Moreover, the convergence is uniformly $C^{k,\alpha}$ on $s$.
 By HT conjecture (cf. \cite{WfZ1}), the limit corresponds to a potential function $\theta_v$ of HVF $v$ on $ M_\infty$. In particular,
 there are two points $x_\infty, y_\infty$ in $\mathfrak a_+'$, both of which are limits of two sequences $\{x_l\},   \{y_l\}$ in $U_k'-x_t'$ such that
 $$|\theta_v(x_\infty)-\theta_v(y_\infty)|\ge 2a_0.$$
 Thus as $l>>1$,
 \begin{align}\label{gap-sequence}
| \frac{\partial \hat\psi'}{\partial s}(x_l)- \frac{\partial \hat\psi'}{\partial s}(y_l)|\ge a_0, ~\forall ~s\in  [l\epsilon_0, (l+1)\epsilon_0].
\end{align}

On the other hand, there is a sequence of constants $c_l$ such that
$$\lim_{s\to l\epsilon_0^{-}}\frac{\partial \psi'}{\partial s}=\frac{\partial\psi'}{\partial s}|_{l\epsilon_0}+c_l.$$
  Then we also have
$$\lim_{s\to l\epsilon_0^{-}} \frac{\partial \hat\psi'}{\partial s}=\frac{\partial \hat\psi'}{\partial s}|_{l\epsilon_0}+c_l'.$$
In particular,
$$\lim_{s\to l\epsilon_0^{-}} \frac{\partial \hat\psi'}{\partial s}(x_l)-\lim_{s\to l\epsilon_0^{-}} \frac{\partial \hat\psi'}{\partial s}(y_l)
=\frac{\partial \hat\psi'}{\partial s}|_{l\epsilon_0}(x_l)-\frac{\partial \hat\psi'}{\partial s}|_{l\epsilon_0}(y_l).$$
Thus the function $(\frac{\partial \hat\psi'}{\partial s}(x_l)-\frac{\partial \hat\psi'}{\partial s}(y_l))$ is continuous at $s$.
  Hence,  for any  two integers $l, k$ with $l<k$,  we get
\begin{align}\label{osc-psi} &[\hat\psi'|_{k\epsilon_0}(x_l)- \hat\psi'|_{l\epsilon_0}(x_l)] -  [\hat\psi'|_{k\epsilon_0}(y_l)- \hat\psi'|_{l\epsilon_0}(y_l)]\notag\\
&=\int_{l\epsilon_0}^{k\epsilon_0}[ \frac{\partial \hat\psi'}{\partial s}(x_l)- \frac{\partial \hat\psi'}{\partial s}(y_l)]dt.
\end{align}
But,  it is impossible since the  term is  always bounded at the left side of  (\ref{osc-psi}) by  (\ref{hat-psi}) while the  term goes to  the infinity as $k\to \infty$  by (\ref{gap-sequence}) at the right  side of  (\ref{osc-psi}). Hence, we prove  (\ref{h-holomorphic}).

Let  $\psi_\infty'$ be a $C^\infty$ limit of $\hat\psi_{t_i}'$ and $\mathbf b$ a limit of vectors $x_{t_i}-y_{t_i}$.  Then it is easy to see that
\begin{align}\label{two-limits}\psi_\infty'(x)=\psi_\infty(x-\mathbf b)+c,
\end{align}
for some constant $c$.
Thus by (\ref{derative}) and (\ref{h-holomorphic}), we have
 $$\lim_{t} \hat h_{t}(\cdot) =e_\infty(\nabla \psi_\infty)+c.$$
On the other hand,  since $e_\infty$ is linear on the torus orbit, there is a vector $Y$ in $\mathfrak a_+'$ which generates  $e_\infty$ such that
 $$  e_\infty(\psi_\infty)=<Y,\nabla\psi_\infty>.$$
 Hence, we get (\ref{ht-limit}).

Step 3.   We prove (\ref{Y-vectors}) and assume that $Y$ is not zero without loss of generality.  In case that  $x_t$ satisfies  (\ref{case-1-1}) and  (\ref{case-2-1}) in Theorem \ref{main-theorem-2}.  We may take a sequence of $x_{t_i}$ with  $t_i= i \epsilon_0$. Then
$$|y_{t_{i+1}}-y_{t_i}|\le N_0.$$
Without loss of generality, we may also assume that
\begin{align}\label{orthogonal-2}
\lim_i <\alpha', y_{t_i}>=\delta_{\alpha'}>0, ~\forall~ \alpha' \in \Phi_+\setminus  \Phi_u.
\end{align}
Since
$$ y_{t_i+1}-y_{t_i}=  | y_{t_i+1}-y_{t_i}| (\frac{Y}{|Y|}+o(1)),$$
   for any $\alpha' \in \Phi_+\setminus  \Phi_u$ we get
\begin{align}
0&=\lim_i |<\alpha',  y_{t_i+1}-y_{t_i}>|\notag\\
&= \lim_i | y_{t_{i+1}}-y_{t_i}|~ | <\alpha',  \frac{Y}{|Y|}+o(1)>|\notag\\
&\ge | <\alpha',  \frac{Y}{|Y|}+o(1)>|.  \notag
\end{align}
Hence (\ref{Y-vectors}) must be true.

\end{proof}
\begin{rem}\label{remark-limit} 1).  It seems  possible to prove that $Y=aX_\infty$ for some $a\neq 0$, where  $X_\infty$ is the limit of $X_t$ in (\ref{x-sequence}). In particular,  $Y$ will be not  zero in Case 3) in Theorem \ref{main-theorem-2}.

2) Note that $\psi'_t$ is just different to a constant with $\psi_t$.  By  (\ref{two-limits}),
the limit  $\hat\phi(x)=\psi_\infty'- <\rho_u,  x>$
is same with   the solution $\phi(x)$ in (\ref{limit-equation}), which is the limit of sequence
\begin{align}\label{final-sequence} \psi_{t_i}(x+x_{t_i}+\mathbf b)- \psi_{t_i} (x_{t_i})-<4\rho_u,  x>+c
\end{align}
for some constant $c$.

\end{rem}

By (\ref{orthogonal}), we can define a subgroup $\mathbf H$  of $\hat{\mathbf G}= {\mathbf G}\times {\mathbf G}$ associated to $X_\infty$ as in (\ref{lie-subalgebra}) in Example \ref{ex-degereration}.  In fact,  we  may modify $X_\infty$ to a rational vector in $\mathfrak a$ so that  (\ref{lie-subalgebra}) is satisfied.   Then $\mathbf N=\hat{\mathbf G}/\mathbf H$ is a horosymmetric space with a parabolic subgroup $\mathbf P$  given by (\ref{parabolic-subg-1}). Thus any $\mathbf K$-invariant metric on  $\mathbf N$ is a form of  (\ref{horosymmetry-LL}) which is determined  by  a convex function $\psi$ on $\mathfrak a$ and the KR  soliton equation on  the $\mathbb Q$-Fano compactification $\overline {\mathbf N}$ of $\mathbf N$ is reduced to  (\ref{KR-soliton-LL}).  As a consequence,  the  metrics of form (\ref{horosymmetry-LL})  determined by the solution $\phi$ in   (\ref{limit-equation}) defines a  KR  soliton on $\mathbf N$. Hence,  to complete the proof of Theorem \ref{main-theorem-2}, we need to show that the metric can be extended to $\overline{\mathbf N}$ which  is just the limit of KR flow (\ref{KRF}).

\begin{rem}\label{horospherical} When $\Phi_u=\Phi_+$, $\mathbf P$ is just a Borel subgroup of $\mathbf G\times \mathbf G$. Thus $\mathbf N$ is a  horospherical  space in this case.

\end{rem}

\subsection{Singular KR soliton solution on $\overline{\mathbf N}$}

 \begin{prop}\label{bar-N-soluiion} Let    $Y\in \mathfrak a_+'$  as in Proposition \ref{limit-solution}.  Then there exists  a singular KR soliton $\omega_{KS}'$  w.r.t. $Y$ on the $\mathbb Q$-Fano horosymmetric variety $\overline{\mathbf N}$ \footnote{We do not know  whether  $\omega_{KS}$ can be  extended to be a singular KR soliton on $\overline{\mathbf N}$ with  its weak K\"ahler potential in  $\mathcal E^1(\overline{\mathbf N} ,-K_{\overline{\mathbf N}})$  directly   from the solution $\phi$ of (\ref{limit-equation}) or not.}.
 \end{prop}

\begin{proof}
 Let $\sigma_t$ be a   family of holomorphism  transformations  by actions  $x\to x+y_t$   in  (\ref{holo-tran}).
Then by Proposition  \ref{limit-solution}  (also see   Remark \ref{remark-limit}-2)),   $\phi(x)$ modulo a constant  is a smooth  limit of
$$ \phi_{t_i}=\psi_{t_i} (x+y_{t_i}+\mathbf a_0)-\psi(y_{t_i})- <\rho_u,  x> .$$
Let
$$W_k=\{z\in M|~\pi(z)=x\in U_k\},$$
where $z$ are locally holomorphic coordinates defined by (\ref{local-coordi}) and $\pi$ is the projection of local holomorphic coordinates to $\mathfrak a_+$.
Hence, we can check the curvature with any derivatives of  the induced metrics
\begin{align}\label{limit-sequence}\sigma_{t_i}^*(\sqrt{-1}\partial\bar{\partial}\phi_{\psi_{t_i}})
\end{align}
 is uniformly bounded on $W_k$.
By the non-collapsing result of Perelman \cite{Pe, ST},  these metrics
converges  locally to a KR soliton $\omega_{KS}$ on the    horosymmetric space $\mathbf N$ constructed in Example \ref{ex-degereration}.
Moreover,  $\omega_{KS}$  is of form (\ref{horosymmetry-LL}) which is determined by the solution $\phi$ of (\ref{limit-equation}).

We claim:
\begin{align}\label{volume-N}
{\rm vol}(\mathbf N,\omega_{KS})={\rm vol}(M,\omega_{0}).
\end{align}

First, we notice that  similar with the case of toric manifolds \cite{WxZ, Zhu1},  from the argument for the upper bound estimate of $m_t$ in Proposition \ref{concentration} (also see \cite[Proposition 4.4]{Del1}), we can get an estimate
$$ \int_{\mathfrak a_+\setminus U_k}e^{w_t} dx\le\frac{C}{e^{m_t}} \sum_{l\ge k+1} \frac{(l+1)^n}{e^l},$$
which goes to zero  by Proposition \ref{concentration} as $k\to\infty$.
Since  $e^{-h_{t}}$ is uniformly bounded,   we obtain
$$ \lim_{k,t\to\infty}\int_{\mathfrak a_+\setminus U_k}  e^{-h_t+w_t} dx=0.$$
By  (\ref{elliptic-family}), it follows that
\begin{align}\label{zero-volume}
  \lim_{k,t\to\infty} \int_{\mathfrak a_+\setminus U_k}\pi (\nabla\psi_{t}){\rm Hess}(\nabla^2 \psi_{t})dx
=0.
 \end{align}

Next, we let  $2P_\infty= {\rm Image}(\nabla\psi_\infty) ({\mathfrak a_+'})$.
 Then  by the metric formula  (\ref{horosymmetry-LL}) and the  local convergence of  $\psi_{t_i}$,  we have
\begin{align}\label{volume-identity}
{\rm vol}(\mathbf N,\omega_{KS})&=C_0\int_{2P_\infty-4\rho_u}\pi (y +4\rho_u) dy\notag\\
&=C_0\int_{2P_\infty}\pi (y) dy\notag\\
& =C_0\int_{{\mathfrak a_+'}} \pi (\nabla\psi_\infty){\rm Hess}(\nabla^2 \psi_\infty) dx\notag\\
&= C_0\lim_k\lim_i \int_{U_k}  \pi (\nabla\psi_{t_i}){\rm Hess}(\nabla^2 \psi_{t_i}) dx.
\end{align}
Thus by  (\ref{zero-volume}) together with  (\ref{volume}), we derive
\begin{align}\label{volune-preserve}
{\rm vol}(\mathbf N,\omega_{KS})&= C_0\int_{2P_+}\pi (y) dy - C_0\lim_k\lim_i \int_{\mathfrak a_+\setminus U_k}\pi (\nabla\psi_{t_i}){\rm Hess}(\nabla^2 \psi_{t_i})dx\notag\\
& = C_0\int_{2P_+}\pi (y) dy= {\rm vol}(M,\omega_{0}).
\end{align}
Hence,    (\ref{volume-N}) is true.

By  (\ref{volume-identity}) and (\ref{volune-preserve}),
we  have
\begin{align}\label{p-volume}\int_{2P_\infty}\pi (y ) dy=\int_{2P_+}\pi (y) dy.
\end{align}
Since $ P_\infty\subseteq  P_+$,  we get
\begin{align}\label{polytope-limit} P_\infty=P_+.
\end{align}
As consequence,   the moment polytope $P'$    associated to  the torus  Lie algebra  $\mathfrak  t'$ and the metric $\omega_{KS}$  on $\mathbf N$ is conjugate   with $P$.  Thus $P'$ satisfies the Delzant condition \cite{Ab}.  In particular,
 $P'$  is  \emph{fine}, which  means  each vertex of $P$ is the intersection of precisely $r$
$(={\rm dim}(P'))$ facets (cf. \cite{Do, LTZ2}).

For any  admissible metric $\omega$ representing $\frac{1}{l} c_1( \overline{\mathbf N},  L^{-l}_{\overline{\mathbf N}})$, by the  $\mathbb C^*$-degeneration,
we have
$${\rm vol}(M,\omega_{0})={\rm vol}(\overline{\mathbf N}, \omega).$$
Then   by (\ref{volume-N}),  we get
$${\rm vol}(\mathbf N,\omega_{KS})={\rm vol}(\overline{\mathbf N}, \omega), $$
which means that  $(\mathbf N, \omega_{KS})$ has the full mass. On the other hand,
as in the computation of volume in (\ref{volume-identity}),  it is easy to see that
$${\rm vol}(\overline{\mathbf N},\omega)=C_0\int_{2P_{\overline{\mathbf N}}^+}\pi (y+4\rho_u) dy,$$
where is  $P_{\overline{\mathbf N}}^+$ is the   quotient of   moment polytope $P_{\overline{\mathbf N}}$   associated to  $\frac{1}{l} c_1( \overline{\mathbf N},  L^{-l}_{\overline{\mathbf N}})$ by the  restricted Weyl group (cf. \cite{Del2}).
Thus analogous to (\ref{p-volume}), we derive
 \begin{align}\label{p-volume-2}\int_{2P_\infty-4\rho_u}\pi (y +4\rho_u) dy=\int_{2P_{\overline{\mathbf N}}^+}\pi (y+4\rho_u) dy.
\end{align}
Since $ P'\subseteq   P_{\overline{\mathbf N}}$,  we conclude that
$$P'=  P_{\overline{\mathbf N}}.$$
Hence,  we prove that  $P_{\overline{\mathbf N}}$   is  \emph{fine}  as same as $P'$ while   $(\mathbf N, \omega_{KS})$ has the full mass.

 Now we can apply  Corollary \ref{extension-KR}  to see that  there is   a singular KR soliton $\omega_{KS}'$ w.r.t.  $Y$ on the $\mathbb Q$-Fano horosymmetric variety $\overline{\mathbf N}$. 

 \end{proof}

 \begin{proof}[Proof of Theorem \ref{main-theorem-2}]
It is clear that the family  of $x_t$ either  exists a  sequence of $t_i$ which is uniformly bounded, or
it is  going to the infinity uniformly.  In the latter case, the family  of $x_t$ either  exists a  sequence of $t_i$ whose $Pj(x_{t_i})$  is uniformly bounded, or  $Pj(x_{t})$   goes to the infinity uniformly. Thus we need to prove  the theorem  in  the three cases divided there, respectively.

Case 1). This case is actually same as one by the continuity method in \cite{WxZ, Del1}.
We can let the initial metric $\omega_0$ in (\ref{reduced-KRflow}) as a  background $\mathbf K\times \mathbf K$-invariant metric given by a  $\mathbf K\times \mathbf K$-invariant function
$\psi_0$ as in (\ref{MA-omega}) such that
$$\omega_0=\sqrt{-1}\partial\bar{\partial}\phi_{\psi_0}.$$
Then as in case of toric manifolds, the K\"ahler potentials $\varphi_{\psi_t}=\phi_{\psi_t}-\phi_{\psi_0}$ for the solution $\psi_t$ of (\ref{reduced-KRflow}) has a uniformly upper bound (cf. \cite[Lemma 3.4]{WxZ}, \cite[Proposition 3.2]{Zhu1}). By the Harnack inequality
(cf. \cite[Proposition 3.1, 5.1]{TZ2}),  we get a uniform $C^0$-bound for $\varphi_{\psi_t}$. Hence, by the regularity of KR flow (\ref{KRF}), we get all  $C^{k,\alpha}$-norms for $\varphi_{\psi_t}$.   As a consequence, the limit $\varphi_\infty$ of $\varphi_{\psi_t}$ will define a KE metric
$$\omega_{KE}=\sqrt{-1}\partial\bar{\partial} (\phi_{\psi_0} +\varphi_\infty)$$
on $(M, J)$.
This part is finished.

Case 2). This case is is  same as Case 1).  Let
$$x_t=a_t+ Pj(x_{t}),$$
 where   $ a_t\in \mathfrak a_c.$
 Then we define a new family of   $\mathbf K\times \mathbf K$-invariant functions by
 $$\tilde \psi_t(\cdot)= \psi_t(\cdot+a_t),$$
 By  (\ref{reduced-KRflow}),  $\tilde \psi_t$ satisfies
  \begin{align}\label{reduced-KRflow-3}
\frac{\partial \psi}{\partial t}(\cdot+a_t)=\log [ \pi (\nabla\tilde\psi){\rm Hess}(\nabla^2 \tilde\psi)]+(\tilde\psi+j).
   \end{align}
By the argument as for toric manifolds \cite[Proposition 4.1]{Zhu1}, we can also get $C^0$-estimate for   K\"ahler potentials as in  Case 1),
$$\varphi_{\tilde\psi_t}=\phi_{\tilde\psi_t}-\phi_{\psi_0}.$$
Moreover,   all  $C^{k,\alpha}$-norms of $\varphi_{\tilde\psi_t}$ are uniformly bounded.
Hence, in this case, the   $\mathbf K\times \mathbf K$-invariant metric
$$ \sqrt{-1}\partial\bar{\partial} (  \phi_{\psi_0}+\varphi_{\tilde\psi_t})$$
 will converge to a KR soliton on $(M,J)$ in sense of K\"ahler potentials  as well as   $\hat h_t$ in Lemma \ref{x-vector}  converges to a potential of HVF after a family of holomorphic transformations by $x\to a_t+x$ (also see the relation (26) in \cite{Zhu1}).

 Case 3). By  3) in Remark \ref{theorem2-remark} and  Remark \ref{horospherical} we need to consider Case 3.2).  Then  both of  (\ref{case-1-1}) and (\ref{case-2-1})
 hold  according to  the proof in Lemma  \ref{general-case}.

By (\ref{volume-N}), we see that 
  $(\mathbf N, \omega_{KS})$  does not lose the volume  of $(M, \omega(t))$. Then
  by HT  conjecture (cf. \cite{Bam, WfZ1}),
the metric completion $(M_\infty, \omega_\infty)$ of $(\mathbf N, \omega_{KS})$  is the limit of KR flow  (\ref{KRF}) in GH topology. Moreover,  $M_\infty $ is  homomorphic to a $\mathbb Q$-Fano variety  $\hat M_\infty$ which  admits a singular KR soliton  w.r.t. the HVF $Y$ \cite{WfZ1}.  It remains to show that $\hat M_\infty$ is biholomorphic to  $\overline{\mathbf N}$.

In case that $(M_\infty, \omega_\infty)$ is a KE metric, i.e., $\omega_{KS}=\omega_{KE}$ is KE.   Then $\hat M_\infty $ is reductive by \cite[Proposition A.1]{WfZ2}. By using the  technique of   partial $C^0$-estimate \cite{WfZ1},
 this exists a  special degeneration from $M$ to   $\hat M_\infty $  via Luna's lemma. Thus  the Mabuchi's K-energy on  $(M, J)$  is bounded below \cite{Li17} (also see \cite[Proposition 5.2]{WfZ2}). As a consequence,   $(M, J)$ is  K-semistable in sense of \cite{Ti1} (cf. \cite{Paul}). On the other hand,  $\overline{\mathbf N}$ is also K-polystable  as same as  $\hat M_\infty $ since  we have known that   it admits a  singular  KE metric $\omega_{KE}'$   by  Proposition \ref{bar-N-soluiion} together with  the fact $Y=0$ \cite{Berm}.  Note that $\overline{\mathbf N}$ is also a limit of  special degeneration of $M$ induced by $X_\infty\in \mathfrak a$. Therefore, by a result of Li-Xu-Wang \cite{LXW},
    $\hat M_\infty$  must be  biholomorphic to $\overline{\mathbf N}$.

    In a special case that  $\overline{\mathbf N}$ is smooth,  $\omega_{KE}'$ is  smooth  (cf. \cite[Theorem  B.1]{BBEGZ}).  Then  $(\overline{\mathbf N}, \omega_{KE}')$ is also  the smooth limit of KR flow (\ref{KRF}) by the argument in the proof of \cite[Theorem 1.4]{WfZ1}.  Thus  by the uniqueness of limits of  KR flow,  we can also prove that $\hat M_\infty$  is biholomorphic to $\overline{\mathbf N}$.

In case that $(M_\infty, \omega_\infty)$ is not  a KE metric.  Namely,  $Y\neq 0$.  Then
 by (\ref{Y-vectors})  we can modify  $Y$ to a rational vector in $\mathfrak a$  as $X_\infty$  so that  (\ref{lie-subalgebra}) or   (\ref{orthogonal}) is satisfied.   Thus,
the $\mathbb Q$-Fano variety $\overline{\mathbf  N}$ is also the limit of  $\mathbb C^*$-degeneration of $(M,J)$  induced  by the soliton HVF $Y$ according to the construction of   subgroup $\mathbf H$ of $\mathbf G\times \mathbf G$   in Example \ref{ex-degereration}.    By a result in \cite{CSW}, there is another  $\mathbb C^*$-degeneration relatively to $Y$ from  $\overline{\mathbf N}$ to $\hat M_\infty$.
Since  both of  $\overline{\mathbf N}$ and $\hat M_\infty$   admit  singular KR solitons  w.r.t.  $Y$,   this new $\mathbb C^*$-degeneration must be trivial by the relatively modified K-stability for KR solitons  \cite{BN, WZZ, HL2}.  Hence,  we prove that $M_\infty$ is biholomorphic to  $\overline{\mathbf N}$.  As a consequence,   $\omega_{KS}$ can be extended to a singular KR soliton w.r.t.  $Y$ on $\overline{\mathbf N}$ and $M_\infty$ can be realized by a  $\mathbb C^*$-degeneration of $(M,J)$  induced by $X_\infty$.   The proof of theorem  is complete.

\end{proof}

\section{ Case of  horosymmetric manifolds}

In this section, we  prove Theorem \ref{main-theorem-horosymmetric}.
 By  (\ref{reduced-KE-equ}), as in the case of Fano $\mathbf G$-manifold,  KR flow    (\ref{KRF})   with a  $\mathbf K$-invariant initial metric  on  a Fano   horosymmetric manifold  $(M,J)$ can be  reduced to  the following parabolic equation of MA type,
  \begin{align}\label{reduced-KE-equ-horosymmetric}
\frac{\partial\psi}{\partial t}=\log [ \pi_0 (\nabla\psi+2\rho_u){\rm Hess}(\nabla^2 \psi)]+ (\psi+j_0),  ~{\rm in}~\mathfrak a_+',
 \end{align}
 where  $\mathfrak a_+'$ is  a  cone  in  $\mathfrak a'$ defined in (\ref{cone-a'}), and
 $$j_0(x)=-\log J_0(x), ~{\rm in}~\mathfrak a_+'.$$

 As in Section 3, we let
 $$w_t'(x)=(\psi+j_0)(x).$$
 Then as in (\ref{xi-condition-2}), there is a family of $x_t$ such that
 \begin{align}\label{xi-condition-2-1}
  \inf_{\mathfrak{a}_+'} w_t'=  w_t'(x_t)=m_t,  ~|m_t|\le C.
\end{align}

Analogous to Theorem  \ref{main-theorem-2}, we can establish a convergence result of KR flow for horosymmetric manifolds.

\begin{theo}\label{main-theorem-horo-2}Let $(M, J)$ be a horosymmetric manifold. Let $x_t$ as chosen in (\ref{xi-condition-2-1}).   Then there are only three cases for KR flow (\ref{reduced-KRflow}) as follows.

Case 1).  There is a sequence of $t_i$ such that
 $$|x_{t_i}|\le C.$$
 Then (\ref{reduced-KRflow}) converges to a  KE metric in sense of K\"ahler potentials. As a consequence, $(M,J)$ is a
 KE manifold.

Case 2). $|x_t|\to  \infty$ as $t\to \infty$ and there is a sequence of $t_i$ such that
 $$|Pj(x_{t_i})|\le C.$$
 Then (\ref{reduced-KRflow}) converges to a  KR soliton in sense of K\"ahler potentials. As a consequence, $(M,J)$  admits a KR soliton.

Case 3).   $|Pj(x_t)|\to  \infty$ as $t\to \infty$. Then there are two subcases:

Case 3.1).
For any $\alpha\in \Phi_s^+$, it holds
\begin{align}\label{general-points-2}<W_\alpha, x_{t}>\to\infty, ~{\rm as}~t\to \infty.
\end{align}
Then (\ref{reduced-KRflow}) converges  locally smoothly to a KR soliton $\omega_{KS}$  on a horospherical space in Cheeger-Gromov topology,  whose completion is the  GH  limit of KR flow  (\ref{KRF}) with  a structure of  $\mathbb Q$-Fano horospherical   variety $ M_\infty$  as a limit of  $\mathbb C^*$-degeneration of $(M,J)$   induced by an element in the Lie algebra of Cartan torus of $M$.  Moreover,   $ M_\infty$ is  same with  a  limit of  $\mathbb C^*$-degeneration of $(M,J)$ induced by the soliton HVF of  $\omega_{KS}$, if  $\omega_{KS}$ is not a KE metric.

Case 3.2).   (\ref{general-points-2}) does not hold.  Then there is  a subset  $\Phi_0$ of $\Phi_s^+$
such that
\begin{align}\label{case-1-2}<\beta, x_{t}>\to \infty ~{\rm as} ~t\to\infty, ~\forall \beta\in   \Phi_0
\end{align}
and
\begin{align}\label{case-2-2}\delta_0\le<\alpha', x_{t}>\le A, ~\forall \alpha'\in  \Phi_s^+\setminus  \Phi_0.
\end{align}
Moreover,   (\ref{reduced-KRflow}) converges   locally smoothly to a KR soliton  $\omega_{KS}$  on a horosymmetric space in Cheeger-Gromov topology as in Case 3.1):   whose completion is the  GH  limit of KR flow  (\ref{KRF}) with  a structure of  $\mathbb Q$-Fano horosymmetric  variety  $ M_\infty$  as a limit of  $\mathbb C^*$-degeneration of $(M,J)$   induced by an element in the Lie algebra of Cartan torus of $M$.;    in  case that  $\omega_{KS}$ is not a KE metric,  $ M_\infty$ is  same with   a  limit of  $\mathbb C^*$-degeneration of $(M,J)$ induced by the soliton HVF of  $\omega_{KS}$.

\end{theo}

The proof  Theorem \ref{main-theorem-horo-2} is almost same with one of  Theorem  \ref{main-theorem-2} and  we need to consider Case 3.2) in  the  theorem.

We define a sequence of convex sets in $\mathfrak a_+'$ by
$$U_k=\{x\in \mathfrak a_+'|~ w_t(x) < m_t'+k+1\}$$
and a family  of convex functions  on $U_k$  by
$$\hat \psi_t' (x)=\psi_t' (x+x_t)-\psi_t' (x_t).$$
Then  analogous to Proposition \ref {local-uniform},
  $\hat\psi_t'$ is uniformly $C^{l,\alpha}$-bounded  on each $U_k$. Moreover,
the set $U_k-x_t$ converges to a cone $\mathfrak a_+''\supset \mathfrak a_+'$   in   $\mathfrak a'$  as $t\to\infty$  by
 $$ \mathfrak a_+''=\{x\in \mathfrak a|~<W_{\alpha'}. x>>0, ~\forall ~\alpha'\in \Phi_s^+\setminus\Phi_u''\},$$
where
\begin{align}\label{unipotential-2}
 \Phi_u''=\Phi_u \cup \Phi_0.
 \end{align}

 We introduce a function on   $ \mathfrak a_+''$ by
\begin{align}\label{j''-function}\mathbf J''(x)\,=\,\prod_{\alpha \in   \Phi_s^+\setminus \Phi_+''} \sinh\alpha(x).
\end{align}
and a vector
$$\rho_u'=\frac{1}{2}\sum_{\beta\in  \Phi_u''} \beta.$$
Then as in the proof of Proposition \ref{limit-solution} (also see Remark \ref{remark-limit}), there are an $\mathbf a_0\in  \mathfrak a_+'$,
  a constant $c$ and
  a  sequence of convex functions
  $$\psi_{t_i}(x+x_{t_i}+\mathbf a_0)- \psi_{t_i} (x_{t_i})-<2\rho_u',  x>+c$$
  which converges  to a solution $\phi$  of equation,
\begin{align}\label{reduced-KR-soliton-2}
 \pi_0(\nabla\phi+ 2\rho_u'){\rm Hess}(\nabla^2 \phi)= J''(x) e^{-\phi-Y(\phi)},
 \end{align}
where   $Y\in \mathfrak a_+''$ and  $\pi_0$ is the function defined by (\ref{pi'}).
Moreover,  $Y$ satisfies
\begin{align}\label{Y-vectors-2}
<\alpha', Y>=0, ~\forall~ \alpha' \in \Phi_s^+\setminus  \Phi_u''.
\end{align}

Analogous to  (\ref{limit-equation}),  the solution of (\ref{reduced-KR-soliton-2}) will define a KR soliton  $\omega_{KS}$ with  form  (\ref{horosymmetry-metric}) on a new horosymmetric space
$\mathbf N'=\mathbf G/\mathbf H'$  of $\mathbb Q$-Fano horosymmetric variety $\overline {\mathbf N'}$, which is   induced by a $\mathbb C^*$-degeneration via  an element $X_\infty$  in the Lie algebra of Cartan subgroup of $\mathbf G$
 as in  Example \ref{ex-degereration}.   In fact,   $X_\infty$ is a limit of $\frac{x_t}{|x_t|}$ in  (\ref{case-1-2}) as in
 Lemma \ref{general-case}.  Moreover,   $\omega_{KS}$ has   the full mass on  $\overline {\mathbf N'}$  as (\ref{volume-N}),  and the  moment polytope    associated to  $\frac{1}{l} c_1( \overline{\mathbf N'},  L^{-l}_{\overline{\mathbf N'}})$  is   \emph{fine}.  Thus   by Corollary \ref{extension-KR}   $\overline{\mathbf N'}$   admits a singular KR soliton w.r.t.  $Y$  analogous  to the   $\mathbb Q$-Fano horosymmetric variety $\overline{\mathbf N}$ in the proof of Theorem \ref{main-theorem-2}.

The Lie algebra of $\mathbf H'$ can be constructed as follows.
 By (\ref{Y-vectors-2}), we can construct a  Lie algebra of new  parabolic subgroup $\mathbf P'$ of $\mathbf G$  as in (\ref{lie-p}),
\begin{align}\label{lie-p-2}\mathfrak p'=\mathfrak t+\sum_{\alpha\in \Phi_{\mathbf L'}}X_\alpha + \sum_{\alpha'\in \Phi_{u}'} X_{-\alpha'},
\end{align}
where $ \Phi_{u}'=\Phi_u''$ and  $\Phi_{\mathbf L'}= \Phi_{\mathbf L}\setminus \Phi_u''\cup (-\Phi_u'').$
The first two parts determine  a Levi subgroup  ${\mathbf L'}^{\sigma}$ of $\mathbf P'$.
Then there is a fixed subgroup ${\mathbf L'}^{\sigma}$ of Levi group $\mathbf L'$ by the involution $\sigma$
 with its root system  is given by
  $$\Phi_{{\mathbf L'}^{\sigma}}=\{\alpha\in \Phi_{\mathbf L'}|~ \sigma(\alpha)=\alpha, ~{\rm or}~-\alpha\}.$$
Thus  the Lie algebra of $\mathbf H'$ can be represented  as,
\begin{align}\label{lie-h-2}\mathfrak h=\mathfrak t_0'  +   \sum_{\beta\in \Phi_{{\mathbf L'}^\sigma}} X_\beta + \sum_{\alpha'\in \Phi_{u}'} X_{-\alpha'},
\end{align}
where $\mathfrak t_0'$ is a subtorus of $\mathfrak t$ fixed by $\sigma$.

\begin{proof}[Completion of proof of Theorem \ref{main-theorem-horo-2}]By the argument above,  we have known that  the  metric of form (\ref{horosymmetry-metric})  determined by the solution $\phi$ of   (\ref{reduced-KR-soliton-2}) defines a  KR  soliton $\omega_{KS}$ on $\mathbf N'$, where the root subsystems $\Phi_{u}, \Phi_s^+ $ are replaced by  $\Phi_{u}', \Phi_s'{^+}=\Phi_s^+\setminus \Phi_u'$, respectively.
 Since  $(\mathbf N', \omega_{KS})$  does not lose the volume  of $(M, \omega(t))$,
  by HT  conjecture (cf. \cite{Bam, WfZ1}),
the metric completion  of $(\mathbf N', \omega_{KS})$  is the limit  $(M_\infty, \omega_{\infty})$ of KR flow  (\ref{KRF}) in GH topology.   Moreover   $M_\infty$ is  homomorphic to a $\mathbb Q$-Fano variety $\hat M_\infty$ and $\omega_{KS}$ can be extended  to a singular KR soliton w.r.t.  $Y$  on  $\hat M_\infty $ \cite{WfZ1}.

 We need to prove  that $M_\infty $ is   biholomorphic to  $\overline{\mathbf N}'$.   In fact,  as in the proof of Theorem \ref{main-theorem-2}, we can divide into two cases:  $\omega_{KS}$ is KE or not KE.   In the first case,  the conclusion   comes  from the  uniqueness result of Han-Li for the minimizers of $H$-invariant  \cite{HL1}.  In the second case,  $Y\neq 0$.   Then we can use  the relatively modified K-stability for KR solitons  w.r.t. $Y$  to conclude that   $M_\infty $ is   biholomorphic to  $\overline{\mathbf N}'$.  Moreover,
 $M_\infty $  is  a limit of  $\mathbb C^*$-degeneration  of $(M,J)$  induced  by the soliton HVF $Y$.  The  proof is finished.

\end{proof}

Theorem \ref{main-theorem-horosymmetric} follows from Theorem \ref{main-theorem-horo-2}. Theorem \ref{main-theorem-horo-2} also implies Theorem \ref{g-flow-horo}.

\begin{proof}[Proof of  Theorem \ref{g-flow-horo}]  Without loss of  generality \cite{ WfZ2, HL1},  we may assume that the initial metric   $\omega_0$ of  (\ref{KRF})  is $\mathbf K$-invariant.  By  Theorem \ref{main-theorem-horo-2} and the assumption in Theorem \ref{g-flow-horo},  it suffices  to consider  Case 3) in  Theorem  \ref{main-theorem-horo-2}.    Then the set  $\Phi_0$ in (\ref{case-1-2})  is  empty or not.   If $\Phi_0$ is  not empty, we can use the argument in  the proof of  \cite[Lemma  4.4]{LTZ1} (also see \cite[Lemma 6.4]{Zhu3})  for  $\mathbf G$-manifolds to show that the curvature along the flow must blow-up.   In fact,  the flow will blow-up  if there is a $\beta\in   \Phi_+$ and a sequence $\{x_{t_i}\}$ such that  (\ref{case-1-2}) holds.  Thus the theorem is true.

 On the other hand, if   $\Phi_0$  is  empty,   the  unipotent root  system $\Phi_{u}'$ and restricted  positive root   system $\Phi_s'{^+}$ of the  limit horosymmetric space $N'$ do not change. This means  that $\mathbf N'$ is same with the original one $\mathbf N$.  Thus $\mathbf N$ admits   a  KR  soliton $\omega_{KS}$ w.r.t.  the soliton HVF $Y$ in (\ref{Y-vectors-2})
 with full mess determined by the solution $\phi$ of   (\ref{reduced-KR-soliton-2}). As an application of  Corollary  \ref{extension-KR},  there exists a  singular KR soliton  w.r.t. $Y$  with its K\"ahler potential $\varphi$  in $\mathcal E^1(M,-K_{M})$  on the variety  compatification  $M$  of ${\mathbf N'}$ (or   ${\mathbf N}$).  By the uniqueness  \cite{BN, Bern,  BBEGZ,  HL2},     this  singular metric  should be  same with the GH limit $\omega_\infty$ of KR flow as a singular KR soliton on $M$ \cite{WfZ1}.  Moreover, we have
$$|\varphi|, ~|Y(\varphi)| \le C.$$
 Note that $M$ is the smooth.   Hence, by the regularity of $\omega_{KS}$ \cite[Lemma 7.2]{WfZ2},   $\omega_{KS}$ is in fact a global smooth KR soliton on $M$.  On the other words,   $(M,J)$ admits a  KR-soliton.  This is  a contradiction  with the assumption of theorem!   Therefore,   $\Phi_0$ could  not be empty and we prove the theorem.

  \end{proof}

\section{Appendix: Singular KR solitons on horosymmetric varieties}

In this appendix, we generalize the  existence results  of KR solitons on horosymmetric manifolds in \cite{Del2, DH} to $\mathbb Q$-Fano  horosymmetric varieties.

 Let   $M$ be a $\mathbb Q$-Fano variety  with  klt-singularities and   $X$  a HVF which can be lifted to one  in  $\mathbb CP^N$ as   $M\subset  \mathbb CP^N$.  We call  a current in $2\pi c_1(M)$ as a  singular KR soliton  w.r.t. $X$ on $M$, if  it  is a smooth metric on  ${\rm Reg} (M)$ which satisfies the KR soliton equation (\ref{KRF}) and  whose  weak K\"ahler potential $\varphi$  belongs to  the  space  of $\mathcal E^1(M,-K_M)$   introduced in \cite{BBEGZ}.  It has been shown in \cite[Theorem 4.2]{LTZ2} that  $\varphi\in \mathcal E^1(M,-K_M)$   in case of Fano $\mathbf G$-compactification varieties
 if and only if its Legendre function $u\in  E_{\mathbf K\times\mathbf K}^1(2P_+)$,  where
 \begin{align}\label{e-1-space}\mathcal E^1_{\mathbf K\times \mathbf K}(2P)\,=\,\{u| ~&\text{$u\ge 0$ is convex, Weyl-invariant on $2P$ which satisfies}
\notag\\
& u(O)=\inf_{2P_+}u=0~ (O\in 2P_+)~{\rm and}~\int_{2P_+} u \pi\,dy<+\infty\}.
\end{align}
 \cite[Theorem 4.2]{LTZ2} can be generalized to  $\mathbb Q$-Fano  horosymmetric varieties $\overline{\mathbf N}$ of homogeneous  space $\mathbf N$  here  while
 the integral  condition  in (\ref{e-1-space}) is replaced by
 $$\int_{2(P_++\rho_u)} u \pi_0(\cdot)dy<+\infty,$$
 where $\pi_0$ is defined by (\ref{pi'}) and  $P_+$ is  replaced by the quotient  space  of  moment polytope $P$ associated to $\frac{1}{l} c_1( \overline{\mathbf N},  L^{-l}_{\overline{\mathbf N}})$ by  the  restricted Weyl group.

 As in  case of  $\mathbf G$-manifolds \cite{LZZ, Del2},  we introduce a barycenter  associated  to HVF $X$ induced by an element in $\mathfrak t'$,
  \begin{align}\label{baracenter}{\rm bar}_X(2P_+)={\frac{\int_{2(P_++\rho_u)}y e^{\theta_X(y)}\pi_0(y)\,dy}{\int_{2(P_++\rho_u)} e^{\theta_X(y)}\pi_0(y) \,dy}},
\end{align}
where $\theta_X(y)$ is a  bounded potential of $X$ associated to an admissible K\"ahler metric induced by the Fubini-Study metric of     $\mathbb CP^N$ and $e^{\theta_X(y)}\pi_0\,dy$ is just the weighted MA measure $e^{\theta_X(y)}\omega_\varphi^n$  under the Legendre transformation.  Moreover, we may normalize $\theta_X(y)$ so that
$$\int_{2(P_++\rho_u)} e^{\theta_X(y)}\pi_0(y) \,dy=\int_{2(P_++\rho_u)} \pi_0(y) \,dy=V_0.$$

  The following is a version of  \cite[Theorem 1.2]{LTZ2}  for  the existence of singular KR solitons  on $\mathbb Q$-Fano  horosymmetric varieties.

\begin{theo}\label{LTZ-generalization}Let $\overline{\mathbf N}$ be a $\mathbb Q$-Fano  horosymmetric variety  with  the associated  \emph{fine}  moment   polytope $P$.   Let
$$\rho=\frac{1}{2}\sum_{\alpha'\in \Phi_s^+\cup\Phi_u} \alpha.$$
 Then $M$ admits a singular KR soliton w.r.t $X$ if
\begin{align}\label{bary}
{\rm bar}_X(2P_+)\,\in\, 2\rho+\Xi,
\end{align}
where $\Xi=\{y\in \mathfrak a^*|~ (\alpha, y)>0, ~\forall~\alpha\in\Phi_s^+ \}$  is the relative interior of the cone generated by $\Phi_s^+$.
\end{theo}

 \cite[Theorem 1.2]{LTZ2}  was proved by the variational method  for the Ding-energy  as done for toric varieties by  Berman-Berndtsson  \cite{Berman-Berndtsson}.   At present, we can also use this method for the modified  Ding-energy  \cite{TZ1, BN, HL2}  to $\mathbb Q$-Fano  horosymmetric varieties.   Here we shall assume  that $P$ in Theorem \ref{LTZ-generalization}  is  \emph{fine} as in  \cite[Theorem 1.2]{LTZ2} for $\mathbf G$-compactification varieties  in order to verify that the Ricci potential of Gullimin metric is bounded above.  But we believe  this assumption can be removed as for  $\mathbf G$-compactification varieties  in \cite[Theorem 8.1]{LTZ2}   by using   a recent  result  of  Han-Li \cite{HL2} to verify the uniformly modified  K-stability via  (\ref{bary}).
  All of these  we will leave the detailed proofs to the reader.

The inverse of  Theorem  \ref{LTZ-generalization}  is also true for  singular  KE metrics on Fano $\mathbf G$-compactification varieties
\cite{Del2, LTZ2}. The result can be generalized to    KR solitons  even just defined on a   horosymmetric space with    full mass by the following lemma.

\begin{lem}\label{necessary-result}Let $\overline{\mathbf N}$ be a $\mathbb Q$-Fano  horosymmetric  variety.   Suppose that there is a KR soliton   $\omega_{KS}$     on  $\mathbf N$  w.r.t.  HVF $X$ induced by an element in  $\mathfrak t'$,  which is a  form  (\ref{horosymmetry-metric})  determined by  a solution $\psi$ of (\ref{reduced-KR-soliton}) with full mass on   $\overline{\mathbf N }$.  \footnote{It is equivalent to ${\rm Image}(\nabla \psi)=P$ by \cite[Lemma 4.5]{LTZ2} since $\psi$ is smooth on $\mathbf N$.}
Then  the associated  moment   polytope $P$ of  $\overline{\mathbf N}$   satisfies  the barycenter condition (\ref{bary})  associated  to  $X$.

\end{lem}

\begin{proof}The proof is almost same  with  one of  \cite[Proposition 5.3]{Del1} for KE-metrics on $\mathbf G$-manifolds.   We can  check  (\ref{bary}) directly as follows.  By  (\ref{reduced-KR-soliton}), we have
\begin{align}
 \pi_0(\nabla\psi+ 2\rho_u){\rm Hess}(\nabla^2 \psi)e^{\theta_X} =  e^{-(\psi+j_0)}, \notag
 \end{align}
 where $j_0=-\log J_0$.
 Let $\psi'= \psi-<2\rho_u,x>$  and $j_0'=j_0- <2\rho_u,x> $.  The the above equation becomes
 \begin{align}
 \pi_0(\nabla\psi'){\rm Hess}(\nabla^2 \psi')e^{\theta_X} =  e^{-(\psi'+j_0')}=e^{-w},\notag
 \end{align}
 where  $w=\psi' + j_0'$.
 Note that
   \begin{align}\label{volume-horo}  \int_{\mathfrak a_+'} \pi_0(\nabla\psi){\rm Hess}(\nabla^2 \psi)  dx=V_0
    \end{align}
Then  the  full mass condition of $\omega_{KS}$, we get
  \begin{align}\label{bartcentor-form}
 \int_{\mathfrak a_+'} \frac{\partial \psi'}{\partial \xi} e^{-w}dx=V_0< {\rm bar}_X(2P_+), \xi>.
  \end{align}
On the other hand, by  the fact that
$$ \frac{\partial j_0}{\partial \xi}< -2<\rho_0, \xi>, $$
where
 $$\rho_0=\frac{1}{2}\sum_{\alpha'\in \Phi_s^+} \alpha', $$
we have
\begin{align}\label{bartcentor-form-1}
\int_{\mathfrak a_+'} \frac{\partial j_0'}{\partial \xi} e^{-w}<-2V_0<\rho, \xi> .
 \end{align}
Thus combining  (\ref{bartcentor-form} ) and  (\ref{bartcentor-form-1}),  we  obtain
  \begin{align}\label{balance-condition}
 &V_0(< {\rm bar}(2P_+), \xi>-2<\rho, \xi>)\notag\\
& >\int_{\mathfrak a_+'} \frac{\partial \psi'}{\partial \xi} e^{-w}dx +\int_{\mathfrak a_+'} \frac{\partial j_0'}{\partial \xi} e^{-w}dx.
\end{align}
However,   by the  full mass condition of $\omega_{KS}$,
  it is easy to see that  for any $\xi\in  \mathfrak t'$,
 \begin{align}
 \int_{\mathfrak a_+'} \frac{\partial w}{\partial \xi} e^{-w} dx =- \int_{\mathfrak a_+'} \frac{\partial }{\partial \xi} e^{-w} dx=0.\notag
 \end{align}
 Namely,
 $$\int_{\mathfrak a_+'} \frac{\partial \psi'}{\partial \xi} e^{-w} dx+\int_{\mathfrak a_+'} \frac{\partial j_0'}{\partial \xi} e^{-w} dx  =0.
 $$
 Hence, we derive
\begin{align}\label{bartcentor-form-2}
 < ({\rm bar}(2P_+)- 2\rho),  \xi> > 0,~\forall ~\xi\in  \mathfrak t'.
 \end{align}
 Clearly, (\ref{bartcentor-form-2})  is equivalent to (\ref{bary}).  The lemma is proved.

\end{proof}

By Theorem \ref{LTZ-generalization} and  Lemma \ref {necessary-result},
 we  prove

\begin{cor}\label{extension-KR}  Let   $\overline{\mathbf N}$  be a $\mathbb Q$-Fano  horosymmetric  variety  with  the associated  \emph{fine} moment   polytope.      Suppose that  there is a KR soliton  $\omega_{KS}$  on   $\mathbf N$  w.r.t.  HVF $X$ induced by an element in  $\mathfrak t'$, which  is a  form  (\ref{horosymmetry-metric} determined by a solution $\psi$ of (\ref{reduced-KR-soliton})   with full mass  on  $\overline{\mathbf N}$.  Then  there exists a  singular KR soliton  w.r.t $X$  on  $\overline{\mathbf N}$  with its K\"ahler potential in $\mathcal E^1(\overline{\mathbf N},-K_{\overline{\mathbf N}})$.

\end{cor}


\vskip20mm


\begin{thebibliography}{10}

\bibitem{Ab} Abrue, M.,  \textit{K\"ahler metrics on toric orbifolds},   \emph{ Jour. Diff. Geom.},  \textbf{58} (2001), 151-187.


\bibitem {Bam} Bamler, R., \textit{Convergence of Ricci flows with bounded scalar curvature},    \emph{ Ann. Math.},  \textbf {188}  (2018), 753-831.



\bibitem {BM} Bando, S. and Mabuchi, T., \textit{Uniqueness of K\"ahler
Einstein metrics modulo connected group actions},  Sendai, 1985,  \emph{ Advanced Studies in Pure Mathematics},
\textbf {10} (1987), 11-40.

\bibitem{Berm} Berman, R.,   K-polystability of Q-Fano varieties admitting K\"{a}hler-Einstein metrics,  \emph{Invent. Math.}, \textbf{130} (1997), 1-37.

\bibitem{Berman-Berndtsson}  Berman  R. and  Berndtsson, B.,  \textit{Real Monge-Amp\`ere equations and K\"ahler-Ricci solitons on toric log Fano varieties},   \emph{ Ann. Fac. Sci. Toulouse Math.}, \textbf{22} (2013),  649-711.


\bibitem{BBEGZ} Berman, R., Boucksom S., Essydieux, P., Guedj, V. and Zeriahi A.,  \textit{K\"{a}hler-Einstein metrics and the K\"{a}hler-Ricci flow on log Fano varieties},  arXiv:1111.7158;    \emph{ J. Reine Angew. Math.}, \textbf{751} (2019),  27-89.


\bibitem{BBJ}  Berman,  R.,  Boucksom,  S.  and Jonsson, M.,   \textit{A variational approach to the Yau-Tian-Donaldson conjecture}, \emph{J, Amer. Math. Soc.},  {\bf 34} (2021), 605-652.


\bibitem {BN} Berman, R. and  Nystrom, D.,  \textit{ Complex optimal transport and the pluripotential theory of  K\"ahler-Ricci solitons },
arXiv: 1401.8264.

 \bibitem{Bern} Berndtsson, B.,  \textit{A Brunn-Minkowski type inequality for Fano manifolds
   and some uniqueness theorems in K\"{a}hler geometry},   \emph{Invent. Math.},  \textbf{200} (2015), 149-200.


\bibitem{Cao} Cao, H.D., \textit{Deformation of K\"ahler metrics to
K\"ahler-Einstein metrics on compact K\"ahler manifolds},   \emph{ Invent. Math.},  \textbf{81} (1985), 359-372.

\bibitem{CDS} Chen, X., Donaldson S.  and Sun, S.,
\textit{K\"ahler-Einstein metrics on Fano manifolds,  I, II, III },
  \emph{J. Amer. Math. Soc.},   \textbf {28 }  (2015), 183-197, 199-234, 235-278.

  \bibitem{CW} Chen, X.    and  Wang, B.,
 \textit{Space of Ricci flows (II)-Part B: Weak compactness of the flows},  \emph{J. Differential Geom.},  \textbf{ 116} (2020), 1-123.

\bibitem{CSW} Chen, X., Sun,  S.  and Wang,  B., \textit{K\"ahler-Ricci flow, K \"ahler-Einstein metric, and K-stability},
\emph{ Geometry and Topology},  \textbf{22} (2018), 3145-3173.

\bibitem{Del1} Delcroix, T.,
\textit{K\"ahler-Einstein metrics on group compactifications}, \emph{Geom. Funct.
Anal.}, \textbf{27} (2017), 78-129.

\bibitem{Del2} Delcroix, T.,
\textit{K-Stability of Fano spherical varieties}, \emph{Ann. Sci. \'Ec, Norm. Sup\'er.},  \textbf{53} (2020), 615-662.

\bibitem{Del3} Delcroix, T.,
\textit{K\"ahler geometry of   horosymmetric  varieties, and application to Mabuchi's K-energy functional},  \emph{ J Reine Angew Math.}, \textbf{763} (2020),  129-199.


\bibitem{DH} Delcroix, T. and   Hultgren, J.,
\textit{Coupled complex Monge-Amp\`ere equations on Fano horosymmetric manifolds},  \emph{J. Math. Pures Appl.},  \textbf{153} (2021), 281-315.

\bibitem{DS} Dervan, R.  and  Sz\'ekelyhidi,  G.,   \textit{K\"ahler-Ricci flow and optimal degenerations}, \emph{ J. Differential Geom.},   \textbf{116} (2020),  no. 1, 187-203.


\bibitem{Do02} Donaldson, S.,    \textit{Scalar curvature and stability of toric varieties},   \emph{ J. Differential Geom.},   \textbf{62} (2002),   289-349.




\bibitem{Do} Donaldson, S., \textit{Interior estimates for solutions of Abreu's equation},   \emph{ Collect. Math.},  \textbf{56} (2005), 103-142.

\bibitem {Fu} Futaki, A.,
\textit{K\"ahler-Einstein metrics and geometric invariants},  \emph{Lecture Notes in Math.},  \textbf{1314} (1987),
Springer-Verlag.



 \bibitem{HL2} Han, J.  and  Li,  C.,  \textit{On the Yau-Tian-Donaldson for generalized K\"ahler-Ricci soliton equations},  arXiv: 2006.00903v1, 2020.

  \bibitem{HL1}  Han, J.  and  Li, C.,   \textit{Algebraic uniqueness of K\"ahler-Ricci flow limits and optimal degenerations of Fano varieties}, arXiv: 2009.01010v1, 2020.


\bibitem{Hel} Helgason, S.,
\textit{Differential Geometry, Lie Groups, and symmetric spaces},
Academic Press, Inc., New York-London, 1978.


\bibitem{Kn} Knapp, A., \textit{Lie Groups beyond an introduction}, Birkh\"auser Boston, Inc., Boston, 2002.

\bibitem{Li17}  Li, C.,  \textit{Yau-Tian-Donaldson correspondence for K-semistable Fano manifolds},  \emph{J. Reine Angew. Math.},  \textbf{733} (2017), 55-85.



\bibitem{Li}  Li,  C.,  \textit{$G$-uniform stability and K\"ahler-Einstein metrics on Fano varieties}, \emph{Invent. Math.},   \textbf{ 227} (2022), 661-744.

\bibitem{LS12} Li, C. and Sun, S.,    \textit{Conic K\"ahler-Einstein metrics revisited},    \emph{Comm. Math. Phys.},   \textbf{331}  (2014), 927-973.


\bibitem{LTW}  Li, C.,   Tian,  G. and  Wang, F.,   \textit{On Yau-Tian-Donaldson conjecture for singular Fano varieties}, \emph{Comm. Pure Appl. Math.},  {\bf 74} (2021), 1748-1800.

 \bibitem{LXW} Li, C.  and Xu,  C. and Wang, X.,  \textit{Algebraicity of the metric tangent cones and  equivariant K-stability},
 \emph{J. Amer., Math., Society},  \textbf{34} (2021), no. 4, 1175-1214.


\bibitem{LL}   Li, Y.   and   Li, Z.,  \textit{ Semistable degenerations of $Q$-Fano  compactifications},  arXiv:2103.06439v3, to appear in  \emph{ Peking J. Math.}


\bibitem{LZZ} Li, Y.,  Zhou, B. and Zhu, X. H., \textit{ K-energy on polarized compactifications of Lie
groups},  \emph{J. of Func. Analysis},  \textbf{275} (2018), 1023-1072.

\bibitem{LTZ1}  Li, Y., Tian, G. and Zhu, X. H., \textit{Singular limits of K\"ahler-Ricci flow on Fano $G$-manifolds}, arXiv:1807.09167,  to appear in   \emph{ Amer. J. Math.}


\bibitem{LTZ2}   Li, Y., Tian, G. and Zhu, X. H.,  \textit{Singular K\"ahler-Einstein metrics on  $\mathbb Q$-Fano compactifications of a Lie group}, \emph{Math. Eng.},   \textbf{5} (2023), no. 2,  Paper No. 028, 43pp ;  arXiv: 2001. 11320.


\bibitem{Paul}   Paul, S.T,   \textit{Hyperdiscriminant polytopes, Chow polytopes, and Mabuchi energy asymptotics},  \emph{Ann. Math.},  \textbf{175} (2012), 255-296.

\bibitem{PS} Podesta F. and Sprio A.,   \textit{K\"ahler-Ricci solitons on homogeneous toric bundle},    \emph{J. Reine Angew. Math.},   \textbf{642}  (2010), 109-127.

\bibitem {Pe} Perelman, G.
\textit{The entropy formula for the Ricci flow and its geometric applications},
arXiv:0211159.

\bibitem{ST}  Sesum, N.   and  Tian, G.,    \textit{Bounding scalar curvature and diameter along the K\"{a}hler-Ricci flow (after Perelman)},    \emph{ J. Inst. Math.   Jussiu},  \textbf{7}(2008),  575-587.


\bibitem {Ti1} Tian, G.,
\textit{K\"ahler-Einstein metrics with positive scalar curvature},
\emph{Invent. Math.}, \textbf{130} (1997), 1-37.

\bibitem{Ti2} Tian, G.,
\textit{K-stability and K\"ahler-Einstein metrics},
\emph{Comm. Pure Appl. Math.}, \textbf{68} (2015), 1085-1156.

\bibitem {TZ1} Tian, G. and Zhu, X. H.,
\textit{Uniqueness of K\"ahler-Ricci solitons},
\emph{Acta Math.},  \textbf{184} (2000), 271-305.

\bibitem{TZ3} Tian, G. and Zhu, X.H.,
\textit{A new holomorphic invariant and uniqueness of K\"ahler-Ricci
solitons},  \emph{Comm. Math. Helv.},  \textbf{77} (2002), 297-325.


\bibitem{TZ2} Tian, G. and Zhu, X. H., \textit{Convergence of the
K\"ahler-Ricci flow},   \emph{J. Amer Math. Sci.},  \textbf{17}
(2006), 675-699.


\bibitem{TZ13} Tian, G. and Zhu, X. H.,  \textit{Convergence of the K\"ahler-Ricci flow on Fano manifolds},   \emph{J.  Reine Angew.  Math.}, \textbf{678} (2013), 223-245.

\bibitem{TZhang} Tian, G. and Zhang, Z.L.,
\textit{Regularity of K\"ahler-Ricci flows on Fano manifolds},   \emph{ Acta Math.},  \textbf{216} (2016), 127-176.

\bibitem{TZZZ} Tian, G.,  Zhang, S., Zhang, Z. and Zhu, X. H., \textit{Perelman's entropy and K\"ahler-Ricci flow an a Fano Manifold},  \emph{Tran. AMS.}, \textbf{365} (2013), 6669-6695.

   \bibitem{WZZ} Wang, F., Zhou, B. and Zhu, X. H., \textit{ Modified Futaki invariant and equivariant Riemann-Roch formula},   \emph{Adv. Math.},  289 (2016), 1205-1235.



 \bibitem{WfZ1} Wang, F.  and   Zhu, X.H.,   \textit{Tian's partial $C^0$-estimate implies Hamilton-Tian's conjecture},   \emph{Adv. Math.},  \textbf{381} (2021),  1-29.

\bibitem{WfZ2} Wang, F.  and   Zhu, X.H.,   \textit{ Uniformly strong convergence of K\"ahler-Ricci flows on a Fano manifold},     \emph{Sci. China Math.},  \textbf{65} (2022),  2335-2370.



\bibitem{WxZ} Wang, X. and Zhu, X. H.,
\textit{K\"ahler-Ricci solitons on toric manifolds with positive first Chern class},
 \emph{Adv. Math.}, \textbf{188} (2004), 87-103.



 \bibitem{Zhq} Zhang, Q.,  \textit{Bounds on volume growth of geodesic balls under Ricci flow},   \emph{ Math. Res. Lett.},
 \textbf{19} (2012),  245-253.

 \bibitem{Zhu4}  Zhu, X. H.,  \textit{K\"ahler-Ricci soliton type equations on compact complex
       manifolds with $C_1(M)>0$},   \emph{ J. Geom. Anal.},  \textbf{10} (2000),  759-774.

\bibitem{Zhu1}  Zhu, X. H.,  \textit{ K\"ahler-Ricci flow on a toric manifold with positive first Chern class.  Differential geometry},  323-336, \emph{ Adv. Lect. Math. (ALM)},  \textbf{22} (2012), Int. Press, Somerville, MA.

\bibitem{Zhu2}  Zhu, X. H.,  \textit{ K\"ahler-Einstein Metrics on Toric manifolds and G-manifolds},   Geometric analysis in honor of Gang Tian's 60th birthday,  \emph{Progr. Math.}, \textbf{333} (2020), Birkhuser/Springer,  545-585.

\bibitem{Zhu3}  Zhu, X. H.,  \textit{ K\"ahler-Ricci flow on Fano manifolds}, ICM, 2022.


\end{thebibliography}
\end{document}